\pdfoutput=1
\documentclass[11pt]{article}

\usepackage[margin=1in]{geometry}
\usepackage{amsmath,amssymb,mathtools,amsthm}
\usepackage{algorithm}
\usepackage{algpseudocode}
\usepackage{xcolor}
\usepackage{graphicx}
\usepackage{subcaption}
\usepackage[T1]{fontenc}
\usepackage{lmodern}
\usepackage[numbers,sort&compress]{natbib}
\usepackage[colorlinks=true,linkcolor=black,citecolor=black,urlcolor=black]{hyperref}
\usepackage{microtype}
\newcounter{lsecaff}
\theoremstyle{plain}
\newtheorem{theorem}{Theorem}[section]
\newtheorem{lemma}[theorem]{Lemma}
\newtheorem{corollary}[theorem]{Corollary}

\theoremstyle{definition}
\newtheorem{definition}[theorem]{Definition}

\theoremstyle{remark}
\newtheorem{remark}{Remark}
\newtheorem*{notation}{Notation}

\newtheorem{assumption}{Assumption}

\newcommand{\ip}[2]{\left\langle #1,#2\right\rangle}

\newcommand{\R}{\mathbb{R}}
\newcommand{\tr}{\mathrm{tr}}

\makeatletter

\newcommand{\Rmnum}[1]{\expandafter\@slowromancap\romannumeral #1@}
\makeatother

\title{Complexity Analysis of the Regular Simplicial Search Method with Reflection and Shrinking Steps for Derivative--Free Optimization}

\author{
	Liyuan Cao\thanks{School of Mathematics, Nanjing University, Nanjing 210093, P.\,R.\,China.
		Email: \texttt{caoliyuan@nju.edu.cn}}
	\and Wei Hu\thanks{State Key Laboratory of Scientific and Engineering Computing (LSEC),
		Institute of Computational Mathematics and Scientific/Engineering Computing (ICMSEC),
		AMSS, Chinese Academy of Sciences, Beijing 100190, P.\,R.\,China.
		Email: \texttt{huwei@amss.ac.cn}}
	\setcounter{lsecaff}{\value{footnote}}
	\and Jinxin Wang\thanks{Chicago Booth School of Business, University of Chicago,
		Chicago, IL 60637, USA. Email: \texttt{wangjinxin68@gmail.com}}
}
\date{August 2025}

\begin{document}
	\maketitle
	
	\begin{abstract}
		Simplex--type methods, such as the well--known Nelder--Mead algorithm, are widely used in derivative--free optimization (DFO), particularly in practice. Despite their popularity, the theoretical understanding of their convergence properties has been limited, and until very recently essentially no worst--case complexity bounds were available. Recently, Cao et al. provided a sharp error bound for linear interpolation and extrapolation and derived a worst--case complexity result for a basic simplex--type method. Motivated by this, we propose a practical and provable algorithm --- the \emph{regular simplicial search method} (RSSM), that incorporates reflection and shrinking steps, akin to the original method of Spendley et al. We establish worst--case complexity bounds in nonconvex, convex, and strongly convex cases. These results provide guarantees on convergence rates and lay the groundwork for future complexity analysis of more advanced simplex--type algorithms.
	\end{abstract}
	
	
	\section{Introduction}\label{sec:intro}
	Derivative--free optimization (DFO) methods are those solving optimization problems without using derivatives of the objective function. It is a common scenario in practice where derivatives are unavailable, inaccessible or expensive to acquire. A typical example is that the objective function value is obtained through a ``black box'' procedure or complex simulation process, see e.g., \cite{Conn2009}, \cite{Audet2017}, \cite{Lar19} for an introduction of DFO. Apart from the finite--difference methods and others using approximate gradients in gradient--based optimization methods (which are usually referred to as \textit{zeroth--order methods}), derivative--free methods are mainly divided into \textit{model--based methods} and \textit{direct search} methods. See \cite{Berahas2022} for a theoretical and empirical comparison of various approximation approaches.
	
	\textbf{Model--based methods} construct model functions to locally approximate the original objective function and utilize the information of model functions to guide the optimization process. Various model--based methods use the trust--region framework and build model functions through polynomial interpolation or radial basis functions~\cite{Wild08}. Powell designed a series of renowned efficient methods like COBYLA~\cite{Pow94}, UOBYQA~\cite{Pow02}, NEWUOA~\cite{Pow06}, BOBYQA~\cite{Pow09}, and LINCOA~\cite{Pow15}. Zhang provided the cross--platform interface PDFO~\cite{PDFO24} for these methods. Furthermore, a modernized implementation with bug fixes is available in the PRIMA package~\cite{PRIMA23}. Underdetermined interpolation contributes to the high efficiency of model--based trust--region methods~\cite{Xie23Comb,Xie24Trans,Xie25Least,Xie25Suf}. Probabilistic models~\cite{Gratton18Trust, Cao24} and subspace techniques~\cite{Zhang25, Xie2023MoSub, Cartis23} are also introduced for large--scale problems. The construction of new model functions was recently discussed in \cite{Xie2025New}.
	
	\textbf{Direct search methods} compare function values directly  without constructing explicit approximate surrogates during the optimization process. They include the directional direct search methods and simplex--type methods. In this paper, we focus on simplex-type methods, which sequentially manipulate and move a simplex move a simplex in $\R^n$ (i.e. $n+1$ affinely independent points in $\R^n$). The earliest simplex--type method, proposed by Spendley et al.~\cite{Spe62}, includes a reflection step (which reflects the worst point through the hyperplane defined by the remaining \( n \) points) and a shrinking step (which moves the worst \( n \) points toward the best point). Nelder and Mead~\cite{Nel65} extended these operations by introducing expansion and contraction steps, enabling the simplex to deform and better capture curvature information. The Nelder--Mead method~\cite{Nel65} is arguably the most widely used DFO algorithm. Its simplicity and empirical effectiveness have made it popular among researchers and engineers. Notably, it is the underlying algorithm of \texttt{fminsearch} in MATLAB. See \cite{Mor09,Rio13} for the practical performance of the Nelder--Mead method and other DFO methods.

	Despite its popularity in applications, theoretical results on the convergence properties of the Nelder--Mead method or other simplex--type methods are relatively scarce compared with other DFO methods. McKinnon~\cite{Mck98} provided an example demonstrating that even for smooth and strongly convex functions in \( \R^2 \), the original Nelder--Mead method may fail to converge to the unique stationary point. Moreover, to the best of our knowledge, no prior work has established convergence rates or complexity bounds for the Nelder--Mead method or its variants, while the complexity of some other direct search methods~\cite{Vicente13, Dodangeh16} and model--based trust--region methods~\cite{Gar16} is typically $\mathcal{O}\left(n^2/\varepsilon^2\right)$ under standard smoothness and nonconvex setting.
	
	This paper aims to analyze the worst--case complexity of a simplex--type method similar to those in~\cite{Spe62,Yu79} in nonconvex and convex smooth assumptions. We leverage the sharp error bound for linear interpolation and extrapolation established by~\cite{Cao23} to provide quantitative estimates on the total decrease in function value. Our main contributions are summarized as follows:
	
	\begin{itemize}
		\item We propose the regular simplicial search method (RSSM) and analyze its worst--case complexity; to our knowledge, this is the first worst--case result for a simplex--type method that includes shrinking steps.
		\item We quantify the effect of each operation on the total function value over the simplex vertices.
		\item Under smoothness and nonconvexity assumption, RSSM finds an $\varepsilon$--stationary point in $\mathcal{O}\left(n^3/\varepsilon^2\right)$ iterations; additionally under the Polyak--\L{}ojasiewicz (PL) condition, the bound improves to $\mathcal{O}\left(n^2/\varepsilon^2\right)$; in convex case, the bound is  $\mathcal{O}\left(n^2/\varepsilon\right)$ in finding an $\varepsilon$--optimal solution and $\mathcal{O}\left(n^2\log (1/\varepsilon)\right)$ under strong convexity assumption.
	\end{itemize}
	
	The rest of the paper is organized as follows. The remainder of Section~\ref{sec:intro} reviews related work and introduces notation. Section~\ref{sec:prelim} presents the main intuition and preliminary error bounds. Section~\ref{sec:algorithm} introduces the algorithm. Section~\ref{sec:complexity} and Section~\ref{sec:convex} develop the worst--case complexity analysis. Section~\ref{sec:conclusion} concludes. 
	
	\subsection{Related Work}
	
	Although the original Nelder--Mead method usually performs well in practice, it has been shown to converge to non--stationary points in certain cases. For example, Woods~\cite{Woo85} provided a counterexample in \( \R^2 \) for a nonconvex function, while McKinnon~\cite{Mck98} constructed a strictly convex function with continuous derivatives that leads to convergence to a non--stationary point. Researchers have been devoted to analyzing the convergence properties of the Nelder--Mead method and developing convergent variants. 
	
	In 1979, Yu~\cite{Yu79} firstly proved the convergence of the simplex--type method of Spendley et al.~\cite{Spe62} for continuously differentiable and lower--bounded functions. In 1998, Lagarias et al.~\cite{Lag98} proved the convergence of the Nelder--Mead method for strictly convex functions in dimension one and provided partial results for dimension two.

		The sufficient descent condition was introduced to guarantee the convergence of the Nelder--Mead method by Kelley in~\cite{Kel99}, based on the simplex gradient (i.e., the gradient of the linear interpolation function over the simplex). It was proved that the method converges assuming the sufficient descent condition holds at every step. Kelley suggested restarting the algorithm when the condition fails, although no theoretical guarantee was provided to ensure the assumed success. 
		
		In 1999, Tseng~\cite{Tse99} proposed a fortified Nelder--Mead algorithm with sufficient descent condtion that manipulates the $m$ worst points simultaneously and is guaranteed to converge to a stationary point.The adaptation also included geometry control as keeping the interior angles of the simplex bounded away from zero, and the simplex radius expands by at most a factor. 
		
		In 2012, Lagarias et al.~\cite{Lag12} improved earlier results by prohibiting expansion steps. They showed that this restricted version can converge to minimizers of any twice--continuously differentiable, lower--bounded function with a positive--definite Hessian. Their work inspire us to consider the case where fewer kinds of operations is involved.
			
		 Rencently, Gal\'{a}ntai~\cite{Gal22} developed a matrix formulation of the Nelder--Mead method and proved its convergence when $n \leq 8$. However, such formulation cannot prove any explicit complexity result.

	\subsection{Notation}
	
	\begin{notation}
		Let \( \Theta_k \) denote the simplex at the $k$--th iteration, consisting of vertices \( \mathbf{x}_1^{(k)}, \dots, \mathbf{x}_{n+1}^{(k)} \). 
		
		To avoid ambiguity, we say that a simplex is \textit{regular} if and only if all its vertices are at the same distance from the centroid (equivalently, all edges have the same length). We use the following notation:
		
		\begin{itemize} 
			\item \( \mathbf{x}_i^{(k)} \): the \( i \)-th vertex of the simplex \( \Theta_k \); 
			\item \( \mathbf{c}_k = \dfrac{1}{n+1} \sum_{i=1}^{n+1} \mathbf{x}_i^{(k)} \): the centroid of \( \Theta_k \);			
			\item $f^\star$: the minimum value of function $f$ (if it could be attained);
			\item $\mathbf{x}^\star$: an optimal solution;
			\item $\hat{f}$: the linear interpolant of $f$ over $\Theta_k$;
			\item $\nabla \hat f(\mathbf{c}_k)$: the simplex gradient (gradient of $\hat{f}$) at the centroid;
			\item $C^{1,1}_L(\R^n)$: $L$--smooth functions in $\R^n$; 
			\item $\mathcal{F}_{0,L}(\R^n)$: convex and $L$--smooth functions in $\R^n$; 
			\item $\mathcal{F}_{\mu,L}(\R^n)$: $\mu$--strongly convex and $L$--smooth functions in $\R^n$;
			\item \( \mathbf{x}_r^{(k)} \): the reflected point,  \( \mathbf{x}_r^{(k)}  = - \mathbf{x}_{n+1}^{(k)} + \dfrac{2}{n}\sum_{i=1}^{n} \mathbf{x}_i^{(k)}  \); \\
			\item \(  \mathbf{x}^{(k)}_{i_s}\): the shrunken points, $i=2,3,...,n+1$, \(  \mathbf{x}^{(k)}_{i_s} = \gamma \mathbf{x}_i^{(k)} + (1-\gamma) \mathbf{x}_1^{(k)} \); \\
			\item \( \delta_k\): the (regular) simplex radius, i.e., the common distance from the centroid to each vertex of \( \Theta_k \);
			\item \( \ell_1(\mathbf{x}), \dots, \ell_{n+1}(\mathbf{x}) \): the Lagrange polynomials on \( \Theta_k \), satisfying
			\[
			\ell_i(\mathbf{x}_j) =
			\begin{cases}
				1, & i = j, \\
				0, & i \neq j.
			\end{cases}
			\]
		\end{itemize}

		We briefly recall the main operations in the Nelder--Mead algorithm. In each non--shrinking iteration \( k \), we rank $\mathbf{x}_1^{(k)}, \mathbf{x}_2^{(k)}, \dots, \mathbf{x}_n^{(k)}, \mathbf{x}_{n+1}^{(k)}$ by their objective function values. The algorithm then replaces the worst vertex \( \mathbf{x}_{n+1}^{(k)} \) with a new trial point \( \mathbf{x}_\alpha^{(k)} \), defined by:
		\[
		\mathbf{x}_\alpha^{(k)} = \dfrac{1}{n}\sum_{i=1}^{n} \mathbf{x}_i^{(k)} + \alpha \left( \dfrac{1}{n}\sum_{i=1}^{n} \mathbf{x}_i^{(k)} - \mathbf{x}_{n+1}^{(k)} \right).
		\]
		The parameter \( \alpha \) determines the type of operation:
		\begin{align*}
			\alpha = 1 &\quad \text{: reflection } (\mathbf{x}_r^{(k)}), \\
			\alpha > 1 &\quad \text{: expansion } (\mathbf{x}_e^{(k)}), \\
			\alpha \in (0,1) &\quad \text{: outside contraction } (\mathbf{x}_{oc}^{(k)}), \\
			\alpha \in (-1,0) &\quad \text{: inside contraction } (\mathbf{x}_{ic}^{(k)}).
		\end{align*}
		
		The update rule for replacing the worst point depends on the function values at trial points:
		\begin{enumerate}
			\item If \( f(\mathbf{x}_r^{(k)}) < f(\mathbf{x}_1^{(k)}) \), replace the worst point with the better of \( \mathbf{x}_r^{(k)} \) and \( \mathbf{x}_e^{(k)} \).
			\item If \( f(\mathbf{x}_1^{(k)}) \leq f(\mathbf{x}_r^{(k)}) < f(\mathbf{x}_n^{(k)}) \), replace it with \( \mathbf{x}_r^{(k)} \).
			\item If \( f(\mathbf{x}_n^{(k)}) \leq f(\mathbf{x}_r^{(k)}) < f(\mathbf{x}_{n+1}^{(k)}) \), replace it with the better of \( \mathbf{x}_r^{(k)} \) and \( \mathbf{x}_{oc}^{(k)} \).
			\item If \( f(\mathbf{x}_{n+1}^{(k)}) \leq f(\mathbf{x}_r^{(k)}) \) and \( f(\mathbf{x}_{ic}^{(k)}) < f(\mathbf{x}_{n+1}^{(k)}) \), replace $\mathbf{x}_{n+1}^{(k)} $with \( \mathbf{x}_{ic}^{(k)} \).
		\end{enumerate}
		
		If none of the above conditions is satisfied, the algorithm performs a shrinking step:
		\[
		\mathbf{x}_i^{(k+1)} \leftarrow \gamma \mathbf{x}_i^{(k)} + (1-\gamma) \mathbf{x}_1^{(k)}, \quad \text{for } i = 2, \dots, n+1, \quad \gamma \in (0,1).
		\]
		
		We further define the following quantities for simplicity:
		\begin{align*}
			v^{(k)} &\triangleq f(\mathbf{x}_{n+1}^{(k)}) - \dfrac{1}{n} \sum_{i=1}^{n} f(\mathbf{x}_i^{(k)})  ~~ \text{(worst–mean gap) }\\ 
			&= \dfrac{n+1}{n} \bigg[f(\mathbf{x}_{n+1}^{(k)}) - \dfrac{1}{n+1} \sum_{i=1}^{n+1} f(\mathbf{x}_i^{(k)})\bigg] ;\\
			v_r^{(k)} &\triangleq f(\mathbf{x}_r^{(k)}) - \dfrac{1}{n} \sum_{i=1}^{n} f(\mathbf{x}_i^{(k)}); \\
			S^{(k)} &\triangleq \sum_{i=1}^{n+1} f(\mathbf{x}_i^{(k)}) ~~\text{(sum of function value);} \\
			\bar{d}^{(k)} &\triangleq \dfrac{1}{n+1} \sum_{i=1}^{n+1} f(\mathbf{x}_i^{(k)}) - f^\star ~~ \text{(average gap);} \\
			d^{(k)} &\triangleq f(\mathbf{c}_k) - f^\star ~~\text{(central gap)}.
		\end{align*}

	\end{notation}
	
	\section{Intuition and Preliminaries}\label{sec:prelim}
	
	\subsection{Intuition}
	
	The main difficulty in analyzing the algorithmic complexity lies in estimating the change in total function value caused by every operation. The geometric shape of the simplex also evolves in an unpredictable manner. We must estimate, using only the function values of current $n{+}1$ vertices, how each operation changes the total function value; this is the only information the algorithm has access to.
	
	Intuitively, the core descending steps are the reflection steps. Expansion and contraction steps are their variants. Note that the results of these operations lie in the same direction, while the ``step sizes'' are different. Shrinking plays the role of a backtracking safeguard, ensuring eventual progress when reflection does not yield sufficient decrease. To quantify the change in total function value, we have to answer two questions:
	
	\begin{itemize}
		\item How to estimate the function value at the candidate point via available function values at vertices?
		\item How can we measure the ``step size'' at each iteration?
	\end{itemize}
	
	If we can precisely control the shape of the simplex during iterations, the second question can be addressed. Specifically, if the simplex remains regular during iterations, the radius of the simplices is predictable.
	
	The first question is relevant to the sharp error bound on linear interpolation and extrapolation, which we introduce below. Extrapolation arises when the query point lies outside the convex hull of the simplex, which is very common in DFO; interpolation is from shrinking while extrapolation is from reflection. Since our algorithm involves these operations, sharp error bounds for both interpolation and extrapolation are esstential.
	
	\subsection{Analytical Error Bound for Linear Interpolation and Extrapolation}
	
	Cao et al.~\cite{Cao23} derived an analytical sharp bound on the function--approximation error of linear interpolation and extrapolation, enabling a quantitative analysis of the change in function values in simplex operations. We briefly introduce their idea here. 
	
	\subsubsection{Error Estimation Problem}
	Cao et al. notice that finding the error bound is equivalent to maximizing the error in a given function class. Restricting that $f$ is quadratic, they derive an analytical lower bound for the maximum error, and demonstrate that it is indeed a sharp error bound under certain conditions, which can be verified through technical calculation.
	
	Still consider an affinely independent interpolation set 
	\[
	\Theta = \{\mathbf{x}_1, \dots, \mathbf{x}_{n+1}\} \subset \R^n.
	\]
	We denote by $\ell_i(\mathbf{x})$ the Lagrange polynomial associated with $\mathbf{x}_i$. At the query point $\mathbf{x}$, we artificially define $\mathbf{x}_0:=\mathbf{x}$ and set $\ell_0(\mathbf{x}_0) = -1$ for simplicity.
	
	For convenience, we abbreviate $\ell_i(\mathbf{x})$ as $\ell_i$ and assume that $\Theta$ is ordered so that
	$\ell_1\ge \ell_2\ge\cdots\ge \ell_{n+1}$. The index sets are defined as:
	\begin{subequations}
		\begin{align*}
			\mathcal{I}_+ &:= \{ i \in \{0,1,\dots,n+1\} : \ell_i > 0 \} = \{1, 2, \dots, |\mathcal{I}_+|\}, \\
			\mathcal{I}_- &:= \{ i \in \{0,1,\dots,n+1\} : \ell_i < 0 \} = \{0, |\mathcal{I}_+|+1, \dots, n+1\}.
		\end{align*}
	\end{subequations}
	By definition $\ell_0 = -1$ implies $\mathcal{I}_- \neq \emptyset$, and we also have $\mathcal{I}_+ \neq \emptyset$.
	
	The analysis depends on the geometry of the interpolation set via the matrix
	\begin{equation}\label{eq:defG}
		G := \sum_{i=0}^{n+1} \ell_i \, \mathbf{x}_i \mathbf{x}_i^\top \in \mathbb{R}^{n \times n}.
	\end{equation}
	
	Namely, for a quadratic function:
	\[
	f(\mathbf{u})=c+v^\top \mathbf{u}+\dfrac12\,\mathbf{u}^\top H \mathbf{u},\qquad H=H^\top\in\mathbb{R}^{n\times n},
	\] 
	it holds that
	\[
	\hat{f}(\mathbf{x}) - f(\mathbf{x}) 
	= \sum_{i=0}^{n+1} \ell_i f(\mathbf{x}_i) 
	= \dfrac{1}{2}\sum_{i=0}^{n+1}\ell_i\,\mathbf{x}_i^\top H \mathbf{x}_i
	= \dfrac{1}{2}\ip{H}{\sum_{i=0}^{n+1}\ell_i \mathbf{x}_i \mathbf{x}_i^\top}
	= \dfrac{1}{2}\ip{H}{G}.
	\]
	
	\subsubsection{Reformulation in Nonconvex Case}
	The problem of maximizing the interpolation error over quadratic functions in $C^{1,1}_L(\R^n)$ can be formulated as the following semidefinite program:
	\[
	\begin{aligned}
		&\max_H && \dfrac{1}{2}\ip{G}{H} \\
		&\text{s.t.} && -L I \preceq H \preceq L I .
	\end{aligned} 
	\]
	Owing to the symmetry of the constraint set, the absolute value in the objective function can be omitted without loss of generality.
	
	To analytically solve it, let the eigendecomposition of $G$ be
	\[
	G = P \Lambda P^\top,
	\]
	where $\Lambda = \mathrm{diag}(\lambda_1, \dots, \lambda_n)$ contains the eigenvalues in non--increasing order. The objective then becomes
	\[
	\dfrac{1}{2}\ip{G}{H}
	= \dfrac{1}{2}\ip{\Lambda}{P^\top H P}.
	\]
	Since $P$ is orthonormal, the constraint $-L I \preceq H \preceq L I$ is equivalent to
	$-L I \preceq P^\top H P \preceq L I$. Given that $\Lambda$ is diagonal, only the
	diagonal entries of $P^\top H P$ affect the objective. Consequently, the optimal $H^\star$ satisfies
	\[
	P^\top H^\star P \;=\; L\,\mathrm{sign}(\Lambda)
	\qquad\Longleftrightarrow\qquad
	H^\star \;=\; L\,P\,\mathrm{sign}(\Lambda)\,P^\top .
	\]
	The corresponding maximal value is therefore
	\[
	\dfrac{1}{2}\ip{G}{H^\star}
	\;=\; \dfrac{L}{2}\sum_{i=1}^n |\lambda_i(G)|
	\;=\; \dfrac{L}{2}\|G\|_\ast,
	\]
	where $\|G\|_\ast$ denotes the nuclear norm (the sum of singular values; for symmetric
	$G$ it equals $\sum_i |\lambda_i(G)|$).
	
	\subsubsection{Reformulation in Convex Case}
	Turning to the case of convex quadratic functions in $\mathcal{F}_{0,L}(\R^n)$, the maximization problem takes the form
	\[
	\begin{aligned}
		&\max_H && \left| \dfrac{1}{2}\ip{G}{H} \right|\\
		&\text{s.t.} && 0 \preceq H \preceq L I .
	\end{aligned} 
	\]
	
	For the positive part of the objective,
	\[
	\max_H \ \dfrac{1}{2}\ip{G}{H},
	\] the optimal solution is given by
	\[
	H^\star \;=\; LP\mathrm{sign}(\max(0, \Lambda))\,P^\top,
	\]
	yielding the maximal value 
	\[
	\dfrac{1}{2}\ip{G}{H^\star}
	\;=\; \dfrac{L}{2}\sum_{i=1}^n \max(0,\lambda_i(G))
	\;=\; \dfrac{L}{2}\tr(G_+).
	\]
	
	For the negative part of the objective,
	\[
	\min_H -\dfrac{1}{2}\ip{G}{H},
	\]
	the minimizer is
	\[
	H^\star \;=\; LP\mathrm{sign}(\max(0, -\Lambda))\,P^\top ,
	\]
	and the corresponding maximum of the absolute value is 
	\[
	\dfrac{1}{2}\ip{G}{H^\star}
	\;=\; \dfrac{L}{2}\sum_{i=1}^n \max(0,-\lambda_i(G))
	\;=\; \dfrac{L}{2}\tr(G_-).
	\]
	
	\subsubsection{Condition of Being a Sharp Bound}
	
	Let $P_+$ and $P_-$ denote the submatrices of $P$ corresponding to the positive and negative eigenvalues of $G$, respectively, and similarly $\Lambda_+$ and $\Lambda_-$ the corresponding diagonal blocks of $\Lambda$.
	
	Let $Y \in \mathbb{R}^{(n+1) \times n}$ be the matrix whose $i$--th row is $(\mathbf{x}_i - \mathbf{x})^\top$ for $i=1,\dots,n+1$, and $Y_+$, $Y_-$ the submatrices of $Y$ formed by the rows indexed by $\mathcal{I}_+$ and $\mathcal{I}_- \setminus \{0\}$, respectively. 
	
	We also use the diagonal weight matrices:
	\begin{align*}
		\mathrm{diag}(\ell) &\in\mathbb{R}^{(n+1)\times(n+1)}
		&&\text{with diagonal }(\ell_1,\ldots,\ell_{n+1}), \\
		\mathrm{diag}(\ell_+) &\in \mathbb{R}^{|\mathcal{I}_+|\times|\mathcal{I}_+|}
		&&\text{containing }\{\ell_i\}_{i\in\mathcal{I}_+}, \\
		\mathrm{diag}(\ell_-) &\in \mathbb{R}^{(|\mathcal{I}_-|-1)\times(|\mathcal{I}_-|-1)}
		&&\text{containing }\{\ell_i\}_{i\in\mathcal{I}_-\setminus\{0\}}.
	\end{align*}
	
	The lemma below in \cite{Cao23} ensures that we can do all the partitioning above.
	
	\begin{lemma}[{\cite[Lemma~5.2, 5.3]{Cao23}}]\label{lem:count-eigs}
		The numbers of positive and negative eigenvalues of $G$ are
		$|\mathcal{I}_+|-1$ and $|\mathcal{I}_-|-1$, respectively. The matrix $Y_- P_-$ is invertible.
	\end{lemma}
	
	We then define
	\begin{equation}\label{eq:defM}
		M := \mathrm{diag}(\ell_+) \, Y_+ P_- \left( Y_- P_- \right)^{-1}.
	\end{equation}
	
	\begin{lemma}[Analytical interpolation error bound {\cite[Theorem 5.1]{Cao23}}]
		Let 
		\[
		\mu_{ij} := e_i^\top M e_{j - |\mathcal{I}_+|}, \quad i \in \mathcal{I}_+, \; j \in \mathcal{I}_- \setminus \{0\},
		\]
		and 
		\[
		\mu_{i0} := \ell_i - \sum_{j \in \mathcal{I}_- \setminus \{0\}} \mu_{ij}, \quad i \in \mathcal{I}_+.
		\]
		If $f \in C_L^{1,1}(\R^n)$ and $\mu_{ij} \ge 0$ for all $(i,j) \in \mathcal{I}_+ \times \mathcal{I}_-$, then
		\[
		\big| \hat{f}(\mathbf{x}) - f(\mathbf{x}) \big| \; \le \; \dfrac{1}{2}\ip{G}{H^\star},
		\]
		and the bound is sharp.
	\end{lemma}
	
	\subsection{Sharp Error Bounds for Simplices}
	Using the above definitions and results, we can establish the following 
	sharp error bounds in regular simplex operations in nonconvex and convex case.
	
	\subsubsection{Nonconvex Error Bound}
	For the reflected point $f(\mathbf{x}_r^{(k)})$, the centroid $f(\mathbf{c}_k)$ and the shrunken points $f(\mathbf{x}^{(k)}_{i_s})$, we have:
	\begin{align}
		&\left| f(\mathbf{x}_r^{(k)}) -\hat{f}(\mathbf{x}_r^{(k)})\right|
		= \left| f(\mathbf{x}_r^{(k)}) + f(\mathbf{x}_{n+1}^{(k)}) 
		- \dfrac{2}{n} \sum_{i=1}^{n} f(\mathbf{x}_i^{(k)}) \right|
		\leq \dfrac{2n+2}{n} L \delta_k^2, \label{eq:refl-err} \\
		&\left| f(\mathbf{c}_k) -\hat{f}(\mathbf{c}_k)\right|
		= \left| f(\mathbf{c}_k) - \dfrac{1}{n+1} \sum_{i=1}^{n+1} f(\mathbf{x}_i^{(k)}) \right|
		\leq \dfrac{L\delta_k^2}{2}, \label{eq:centroid-err}\\
		&\left| f\!\left(\mathbf{x}^{(k)}_{i_s}\right)
		-\hat{f}\left(\mathbf{x}^{(k)}_{i_s}\right)\right| \leq L \dfrac{n+1}{n} \gamma(1 - \gamma) \delta_k^2,
		\quad i=2,\dots,n+1. \label{eq:shrink-err}
	\end{align}
	
	\subsubsection{Convex Error Bound}
	For the reflected point $f(\mathbf{x}_r^{(k)})$, the centroid $f(\mathbf{c}_k)$, we have:
	\begin{align}
		&\left| f(\mathbf{x}_r^{(k)}) -\hat{f}(\mathbf{x}_r^{(k)})\right|
		= \left| f(\mathbf{x}_r^{(k)}) + f(\mathbf{x}_{n+1}^{(k)}) 
		- \dfrac{2}{n} \sum_{i=1}^{n} f(\mathbf{x}_i^{(k)}) \right|
		\leq \left(1+\dfrac{1}{n}\right)^2 L \delta_k^2, \label{eq:refl-err-cvx} \\
		&\left| f(\mathbf{c}_k) -\hat{f}(\mathbf{c}_k)\right|
		=  \dfrac{1}{n+1} \sum_{i=1}^{n+1} f(\mathbf{x}_i^{(k)})-f(\mathbf{c}_k)   
		\leq \dfrac{L\delta_k^2}{2}. \label{eq:centroid-err-cvx}
	\end{align}
	
	For regular simplices, the coefficients satisfy $\mu_{ij}\ge0$ for reflection, centroid and shrinking queries; see Appendix~\hyperlink{app:reflection}{A.3}. Most of the proofs are omitted here and can be found in \cite{Cao23}. 
	We also provide a self--contained detailed exposition in the appendix on how to compute 
	the matrices $G$ (defined in \eqref{eq:defG})
	and $M$ (defined in \eqref{eq:defM}) for each bound.
	
	Cao et al.~\cite{Cao23} applied these bounds to obtain a decrease estimate for the following Algorithm~\ref{alg1}, which is a basic simplex--type method containing only reflection steps.
	
	\begin{algorithm}
		\caption{Simplex-Type Method with Reflection}
		\label{alg1}
		\begin{algorithmic}[1]
			\Require A starting point \( \mathbf{c}_0 \in \R^n \), initial simplex radius \( \delta_0 \).
			\State Construct a regular simplex centered at \( \mathbf{c}_0 \), where each vertex has distance \( \delta_0 \) from \( \mathbf{c}_0 \). Let the initial simplex be \( \Theta_0 \subset \R^n \), consisting of \( n+1 \) vertices.
			\For{\( k = 0,1,2,\dots \)}
			\State Sort the vertices in \( \Theta_k \) based on function values:
			\[
			f(\mathbf{x}_1^{(k)}) \leq f(\mathbf{x}_2^{(k)}) \leq \dots \leq f(\mathbf{x}_{n+1}^{(k)}).
			\]
			\State Compute the reflection point:	
			\[
			\mathbf{x}_r^{(k)} = -\mathbf{x}_{n+1}^{(k)} + \dfrac{2}{n} \sum_{i=1}^{n} \mathbf{x}_i^{(k)}.
			\]
			\State Evaluate \( f(\mathbf{x}_r^{(k)}) \).
			\State Update the simplex:
			\[
			\Theta_{k+1} \gets \Theta_k \setminus \{\mathbf{x}_{n+1}^{(k)}\} \cup \{\mathbf{x}_r^{(k)}\}.
			\]
			\EndFor
		\end{algorithmic}
	\end{algorithm}
	
	Note that the simplex remains regular throughout the iterations, since the isometric reflection preserves its shape. As a result, no geometric distortion occurs, and the distances used in interpolation remain constant.
	
	Cao et al.~\cite{Cao23}, in their Theorem 8.5, further showed that if the initial simplex size is set as \( \delta_0 = \dfrac{2\varepsilon}{5Ln} \), then the number of iterations required to find a point \( \mathbf{c}_k \) satisfying \( \|\nabla f(\mathbf{c}_k)\| \leq \varepsilon \) is bounded by $\mathcal{O} \left( {n^3/\varepsilon^2} \right)$.
	
	\subsection{Assumptions and Preliminaries}
	We consider the noise-free, unconstrained setting in this work and only assume access to function evaluations of $f$, which is standard in DFO.
	
	The following Assumption \ref{as1} holds throughout the paper.
	\begin{assumption}\label{as1}
		We assume that $f(\mathbf{x})$ is $L$--smooth on \( \R^n \), i.e.,
		\[
		\| \nabla f(\mathbf{x}) - \nabla f(\mathbf{y}) \| \leq L \| \mathbf{x} - \mathbf{y} \|, \quad \forall \mathbf{x}, \mathbf{y} \in \R^n.
		\]
		The minimum of $f(\mathbf{x})$ exists and is equal to $f^\star$.
	\end{assumption}

	The following inequalities in \cite{Cao23} are useful in our analysis, 
	\begin{align}
		&\| \nabla f(\mathbf{c}_k) - \nabla \hat{f}(\mathbf{c}_k) \|^2 \leq \dfrac{n}{4} L^2 \delta_k^2,\label{eq:gra-error}\\
		&\|\nabla \hat{f}(\mathbf{c}_k)\| \leq \dfrac{n}{\delta_k} 
		\left( f(\mathbf{x}_{n+1}^{(k)}) - \dfrac{1}{n+1} \sum_{i=1}^{n+1} f(\mathbf{x}_i^{(k)}) \right)=\dfrac{n}{\delta_k}\dfrac{n}{n+1}v^{(k)}, \\
		&f(\mathbf{x}_{n+1}^{(k)}) - \dfrac{1}{n+1} \sum_{i=1}^{n+1} f(\mathbf{x}_i^{(k)}) = \dfrac{n}{n+1}v^{(k)}
		\geq \dfrac{\delta_k}{n} \left( \|\nabla f(\mathbf{c}_k)\| - \dfrac{\sqrt{n}}{2} L \delta_k \right).\label{eq:grad-gap-lower}
	\end{align}
	
	\section{Our Algorithm}\label{sec:algorithm}
	
	The results in~\cite{Cao23} are promising, and provide a pathway for analyzing the complexity of simplex--type methods. Nevertheless, the convergence of their simple algorithm relies on a sufficiently small $\delta_0$, which is not practical (e.g. $\varepsilon={10}^{-5}$ and $n\approx 10^2$). The introduction of shrinking steps and variable radii $\delta_k$ is essential for practical simplex--type methods.
	
	To this end, we propose Algorithm~\ref{alg2}, which augments Algorithm~\ref{alg1} by incorporating a shrinking operation, in the spirit of the simplex--type methods in~\cite{Spe62,Yu79}. Importantly, the simplex remains regular throughout the iterations.
	
	\begin{algorithm}
		\caption{Regular Simplicial Search Method}
		\small
		\label{alg2}
		\begin{algorithmic}[1]
			\Require A starting point \( \mathbf{c}_0 \in \R^n \), shrinking parameter \( \gamma \in (0,1) \), initial simplex radius \( \delta_0 \), parameter $\beta>0$, and Lipschitz constant \( L \).
			\State Construct a regular simplex centered at \( \mathbf{c}_0 \), where each vertex has distance \( \delta_0 \) from \( \mathbf{c}_0 \). Let the initial simplex be \( \Theta_0 \subset \R^n \), consisting of \( n+1 \) vertices.
			\For{\( k = 0,1,2,\dots \)}
			\State Sort the vertices in \( \Theta_k \) based on function values:
			\[
			f(\mathbf{x}_1^{(k)}) \leq f(\mathbf{x}_2^{(k)}) \leq \dots \leq f(\mathbf{x}_{n+1}^{(k)}).
			\]
			\State Compute the reflection point:	
			\[
			\mathbf{x}_r^{(k)} = -\mathbf{x}_{n+1}^{(k)} + \dfrac{2}{n} \sum_{i=1}^{n} \mathbf{x}_i^{(k)}.
			\]
			\State Evaluate \( f(\mathbf{x}_r^{(k)}) \).
			\If{ \( f(\mathbf{x}_r^{(k)}) - f(\mathbf{x}_{n+1}^{(k)}) \leq -\dfrac{2n+2}{n}\beta L\delta_k^2 ~\)}
			\State Accept \( \mathbf{x}_r^{(k)} \) and update the simplex:
			\[
			\Theta_{k+1} \gets \Theta_k \setminus \{\mathbf{x}_{n+1}^{(k)}\} \cup \{\mathbf{x}_r^{(k)}\}.
			\]
			\State Maintain the simplex radius: \(\delta_{k+1} = \delta_k.\)
			\Else
			\State Shrink the simplex towards \( \mathbf{x}_1^{(k)} \):
			\[
			\mathbf{x}_i^{(k+1)} \gets \gamma \mathbf{x}_i^{(k)} + (1-\gamma) \mathbf{x}_1^{(k)}, \quad i=2,\dots,n+1.
			\]
			\State Update the simplex radius: \(\delta_{k+1} = \gamma \delta_k.\)
			\EndIf
			\EndFor
		\end{algorithmic}
	\end{algorithm}
	\begin{remark}
		Algorithm~\ref{alg2} maintains the regularity of the simplex via isometric reflection and uniform shrinking. The update rule accepts a reflection step only when it yields a sufficient decrease in terms of function values, quantified in terms of the simplex radius \( \delta_k^2 \). Otherwise, a shrinking step is applied. 
	\end{remark}
	
	\begin{remark}
		Although Algorithm~\ref{alg2} involves $L$ and $\beta$, they only serve as \textit{theoretical analysis scaffolds}. In practical implication, one does \emph{not} need to know \(L\) and may replace the coefficient 
		by any user--chosen positive number \(\eta>0\) as
		\[
		f(\mathbf{x}_r^{(k)})-f(\mathbf{x}_{n+1}^{(k)}) \;\le\; -\eta\,\delta_k^2 .
		\]
		From \eqref{eq:gra-error} it is shown that $\|\nabla \hat{f}(\mathbf{c}_k)\|$ closely approximates $\|\nabla f(\mathbf{c}_k)\|$ when $\delta_k$ is small, thus $\|\nabla\hat f(\mathbf{c}_k)\|\le\varepsilon$ can be used as a computable termination criterion in practice.
	\end{remark}
	
	For convenience in the subsequent analysis, we introduce the following reformulations. A new candidate point \( \mathbf{x}_r^{(k)} \) is accepted in Algorithm~\ref{alg2} if and only if it satisfies the sufficient descent condition:
	\begin{equation*}
		f(\mathbf{x}_r^{(k)}) - f(\mathbf{x}_{n+1}^{(k)}) \leq -\dfrac{2n+2}{n}\beta L \delta_k^2.
	\end{equation*}
	
	This condition can be rewritten as:
	\begin{equation}\label{eq:suff-decrease}
		v_r^{(k)} - v^{(k)} \leq -\dfrac{2n+2}{n}\beta L \delta_k^2.
	\end{equation}
	
	Additionally, the extrapolation error bound for reflection \eqref{eq:refl-err} is equivalent to:
	\begin{equation}\label{eq:gr-plus-bound}
		\left| v_r^{(k)} + v^{(k)} \right| \leq \dfrac{2n+2}{n} L \delta_k^2.
	\end{equation}
	
	\section{Complexity on Nonconvex Smooth Functions}\label{sec:complexity}
	For nonconvex smooth functions, we analyze the complexity of Algorithm~\ref{alg2} in finding an $\varepsilon$--stationary point.
	
	Let $N_\varepsilon$ be the first iteration such that the algorithm terminates. It is the first iteration that the norm of the gradient at the centroid drops below $\varepsilon$, i.e., 
	\[  N_\varepsilon = \arg\min_k \left\{k \in \mathbb{N} ~\middle|~ \|\nabla f(\mathbf{c}_k)\| \le \varepsilon \right\}. 
	\] 
	\subsection{Worst--Case Complexity under General Nonconvexity Assumption}
	
	\subsubsection{Sufficient Condition for Reflection (Nonconvex Case)}
	
	We begin by establishing that \( \mathbf{x}_r^{(k)} \) is always accepted when the simplex radius \( \delta_k \) is sufficiently small.
	
	\begin{lemma}[Sufficient Condition for Reflection] \label{lem:sufficient-reflect}
		Suppose that Assumption~\ref{as1} hold. 
		If 
		\begin{equation}\label{eq:ncvsd}
			\delta_k \leq \dfrac{\|\nabla f(\mathbf{c}_k)\|}{L\kappa^1_{n,\beta}}, ~~\text{where} ~~ \kappa^1_{n,\beta}:= (\beta+1)n + \dfrac{\sqrt{n}}{2},
		\end{equation}
		then $\mathbf{x}_r^{(k)}$ is accepted and iteration $k$ is a reflection step. 
		Furthermore,  
		\begin{equation}\label{eq:ncvglb}
			\delta_k \geq \bar{\delta} := \dfrac{\gamma \varepsilon}{L\kappa^1_{n,\beta}}, 
			\quad \forall k \in \{0, 1, \dots, N_\varepsilon-1\} 
		\end{equation}
		as long as $\delta_0 \geq \bar{\delta}$. 
	\end{lemma}
	
	\begin{proof}
		Suppose that \( \mathbf{x}_r^{(k)} \) is rejected at iteration \( k \). 
		Then, by combining  \eqref{eq:grad-gap-lower}, \eqref{eq:suff-decrease}, \eqref{eq:gr-plus-bound}, 
		\( \delta_k \) must satisfy the following inequality
		\begin{equation}\label{eq:delta-threshold}
			\left( \beta n + n + \dfrac{\sqrt{n}}{2} \right) L \delta_k \geq \|\nabla f(\mathbf{c}_k)\|.
		\end{equation}
		When \( \delta_k \leq \dfrac{\|\nabla f(\mathbf{c}_k)\|}{L\kappa^1_{n,\beta}} \), inequality \eqref{eq:delta-threshold} does not hold,  and thus \( \mathbf{x}_r^{(k)} \) must be accepted.
		
		We prove \eqref{eq:ncvglb} by induction. If \( \delta_k \geq \bar{\delta} = \dfrac{\gamma \varepsilon}{L\kappa^1_{n,\beta}} \) for some \( k \in \{0, \dots, N_\varepsilon -1\} \), then $\delta_{k+1} \geq \bar{\delta}$ if $k$ is a reflection step. 
		
		By \eqref{eq:ncvsd}, $\delta_k >  \dfrac{\varepsilon}{L\kappa^1_{n,\beta}}$ if $k$ is a shrinking step, thus $\delta_{k+1} \geq \bar{\delta}$ holds for both cases. Since $\delta_{0} \geq \bar{\delta}$, by induction  \eqref{eq:ncvglb} holds for all \( k \in \{0, \dots, N_\varepsilon\} \).
	\end{proof}
	
	\subsubsection{Function Value Change due to Reflection and Shrinking Steps}
	\begin{lemma}[Minimum Function Decrease due to a Reflection Step]\label{lem:min-decrease-reflection}
		Suppose that Assumption~\ref{as1} holds, $\delta_0 > \bar{\delta} $ and the stopping condition is not satisfied yet. If the reflection step \( k \) is successful for some \( k \in \mathbb{N} \), then:
		\begin{equation}\label{eq:ref-step-decrease}
			S^{(k+1)} - S^{(k)} \leq -(n+1)\dfrac{2\beta \gamma^2 \varepsilon^2}{L n (\kappa^1_{n,\beta})^2 }.
		\end{equation}
	\end{lemma}
	\begin{proof}
		In every reflection step, the sufficient decrease condition always holds. Meanwhile Lemma~\ref{lem:sufficient-reflect} implies that $\delta_k > \bar{\delta} =\dfrac{\gamma \varepsilon}{L\kappa^1_{n,\beta}}, 
		\quad \forall k \in \{0, 1, \dots, N_\varepsilon -1\}.$ Thus:
		\[
		S^{(k+1)} - S^{(k)} = f(\mathbf{x}_r^{(k)}) - f(\mathbf{x}_{n+1}^{(k)}) \leq -\dfrac{2n+2}{n}\beta L \delta_k^2 < -(n+1)\dfrac{2\beta \gamma^2 \varepsilon^2}{Ln(\kappa^1_{n,\beta})^2  }.
		\]
	\end{proof}
	
	\begin{lemma}[Maximum Function Increase due to a Shrinking Step]\label{lem:max-increase-shrink}
		Suppose that Assumption~\ref{as1} holds, $\delta_0 > \bar{\delta} $ and the stopping condition is not satisfied yet. If the reflection step $k$ fails for some \( k \in \mathbb{N} \), then:
		\begin{equation}\label{eq:shrink-step-increase}
			S^{(k+1)} - S^{(k)} \leq L\gamma(1-\gamma)(n+1)\delta_k^2.
		\end{equation}
		
		\begin{proof}
			In the shrinking step \(k\), \(\mathbf{x}_i^{(k+1)} = \gamma \mathbf{x}_i^{(k)} + (1 - \gamma) \mathbf{x}_1^{(k)}\), \(i = 2, \dots, n+1\).
			For \(i=2,\ldots,n+1\), the interpolation error bound in \eqref{eq:shrink-err} implies that
			\[
			f(\mathbf{x}_i^{(k+1)})\le 
			\gamma f(\mathbf{x}_i^{(k)}) + (1 - \gamma) f(\mathbf{x}_1^{(k)}) + \dfrac{n+1}{n} L \gamma(1 - \gamma) \delta_k^2.
			\]
			Summing over \( i = 1, \dots, n+1 \), we obtain
			\begin{align*}
				S^{(k+1)} &= f(\mathbf{x}_1^{(k)})+ \sum_{i=2}^{n+1}f(\mathbf{x}_i^{(k+1)}) \\
				&\le (1-\gamma)(n+1)f(\mathbf{x}_1^{(k)}) + \gamma S^{(k)} + (n+1)L \gamma(1 - \gamma) \delta_k^2 \\
				&\le S^{(k)} + (n+1)L\gamma(1-\gamma)\delta_k^2.
			\end{align*}
		\end{proof}
	\end{lemma}  
	
	\begin{lemma}[Number of Shrinking Steps]\label{lem:num-shrink-steps}
		Suppose that Assumption~\ref{as1} holds, $\delta_0 > \bar{\delta} $ and the stopping condition is not satisfied yet. Define \( N_r \) as the number of reflection steps within \( N_\varepsilon \) iterations, and \( N_s \) as the number of shrinking steps. Then, we have the bound:
		\begin{equation}\label{eq:num-shrink-steps}
			N_s < \dfrac{\log(\bar{\delta}/\delta_0)}{\log \gamma}. 
		\end{equation}
	\end{lemma}
	
	\begin{proof}
		Lemma~\ref{lem:sufficient-reflect} indicates that $\delta_k > \bar{\delta}$, $\forall k \leq N_\varepsilon - 1$. Namely, $\delta_0 \gamma^{N_s} > \bar{\delta}$, i.e.
		
		\[
		N_s < \dfrac{\log(\bar{\delta}/\delta_0)}{\log \gamma}.
		\]
	\end{proof}
	
	\begin{lemma}[Maximum Ascent due to All Shrinking Steps]\label{lem:max-ascent-all-shrink}
		Suppose that Assumption~\ref{as1} holds, $\delta_0 > \bar{\delta} $ and the stopping condition is not satisfied yet. The maximum possible ascent of total function value due to all shrinking steps is at most \( (n+1)\psi_0\), where 
		\[
		\psi_0 := \dfrac{L \gamma}{1 + \gamma}\delta_0^2.
		\] 
		\begin{proof}
			The maximum ascent of total function value due to all shrinking steps is
			\begin{align*}
				L \gamma (1 - \gamma) (n+1)\sum_{k=0}^{N_s-1} (\gamma^k \delta_0)^2 &\leq L \gamma (1 - \gamma)  (n+1)\delta_0^2 \sum_{k=0}^{\infty} \gamma^{2k} \\
				&= L \gamma (1 - \gamma)  (n+1)\delta_0^2 \dfrac{1}{1 - \gamma^2} = \dfrac{L \gamma}{1 + \gamma} (n+1) \delta_0^2.
			\end{align*}
		\end{proof}
	\end{lemma}
	
	\subsubsection{Worst--Case Complexity (\Rmnum{1})}
	\begin{theorem}[General Worst--Case Complexity]\label{thm:general-wc}
		Suppose that Assumption~\ref{as1} holds, $\delta_0 > \bar{\delta} $ and the stopping condition is not satisfied yet. For Algorithm~\ref{alg2}, the number of iterations in which \( \|\nabla f(\mathbf{c}_k)\| > \varepsilon \) is bounded above by:
		\begin{equation}\label{eq:general-wc}
			\dfrac{ L( \bar{d}_0+\psi_0) n (\kappa^1_{n,\beta})^{2} }{2\beta \gamma^2 \varepsilon^2} + \dfrac{\log(\bar{\delta}/\delta_0)}{\log \gamma} = \mathcal{O}\left(\dfrac{n^3}{\varepsilon^2}\right),
		\end{equation}
		where
		\[
		\bar{d}_0 := \dfrac{1}{n+1} \sum_{u \in \Theta_0} f(u) - f^\star.
		\]
	\end{theorem}
	\begin{proof}
		Assume that among these $N$ iterations, the number of reflection steps and shrinking steps are $N_r$ and $N_s$, respectively. Due to Lemma~\ref{lem:num-shrink-steps}, we only need to bound $N_r$.
		Applying Lemma~\ref{lem:sufficient-reflect} and Lemma~\ref{lem:max-increase-shrink}, we conclude that
		\begin{align*}
			N_r \cdot (n+1)\dfrac{2\beta \gamma^2 \varepsilon^2}{L n (\kappa^1_{n,\beta})^2 } 
			&\leq S^{(0)}-S^{(N_\varepsilon -1)} +(n+1)\psi_0 \\
			&\leq S^{(0)}-(n+1)f^\star+(n+1)\psi_0 \\
			&= (n+1) (\bar{d}_0+\psi_0).
		\end{align*}
		Solving for $N_r$, we have $N_r\leq \dfrac{L( \bar{d}_0+\psi_0)n({\kappa^1_{n,\beta}})^2}{2\beta \gamma^2 \varepsilon^2}$. Combining with Lemma~\ref{lem:num-shrink-steps}, we have	
		\[
		N_\varepsilon = N_r+N_s \leq \dfrac{L( \bar{d}_0+\psi_0)n ({\kappa^1_{n,\beta}})^2 }{2\beta \gamma^2 \varepsilon^2} + \dfrac{\log(\bar{\delta}/\delta_0)}{\log \gamma}.
		\]
		This completes the proof.
	\end{proof}
	
	\begin{remark}
		Notice that $\kappa^1_{n,\beta} = (\beta+1)n +\sqrt{n}/{2} = \mathcal{O}(n)$ implies the $\mathcal{O}(n^3/\varepsilon^2)$ result if $\beta = \mathcal{O}(1)$.
	\end{remark}
	
	\begin{remark}
		The result above is actually \textit{iteration complexity}, while in DFO we also care about \textit{evaluation complexity}, which is the number of function value evaluations spent in the optimization process. Each reflection step uses only one function evaluation, while a shrinking step requires evaluating $n$ new vertices. Hence the total number of function evaluations is $N_r + n N_s$, which is still of the same order as the iteration complexity, up to an additive $\mathcal{O}(n\log(\delta_0/\bar\delta))$ term from the shrinking steps. Other results follow from the same argument.
	\end{remark}
	
	\subsection{Worst--Case Complexity under the PL Condition}
		\begin{assumption}\label{as2}
		In addition to Assumption~\ref{as1}, we further assume that \( f \) satisfies the Polyak--\L{}ojasiewicz (PL) condition. Namely, there exists a constant \( \mu > 0 \) such that:
		\begin{equation}\label{eq:PL}
			\dfrac{1}{2} \|\nabla f(\mathbf{x})\|^2 \geq \mu \bigl(f(\mathbf{x}) - f^\star\bigr), \quad \forall \mathbf{x} \in \R^n.
		\end{equation}
	\end{assumption}
	
	Under Assumption 2, we are able to improve the worst--case complexity to  $\mathcal{O}({n^2/\varepsilon^2})$ under the condition $\varepsilon < {1}/{n}$, which is reasonable since usually $\varepsilon < 10^{-5}$ and $n<10^3$.

	\subsubsection{Worst--Case Complexity (\Rmnum{2})}
	\begin{theorem}[Improved Worst--Case Complexity]\label{thm:pl-wc}
		Suppose that Assumption~\ref{as2} holds, $\delta_0 > \bar{\delta} $, the stopping condition is not satisfied yet and let \( \varepsilon < {1}/{n}\). Then the number of iterations in which \( \|\nabla f(\mathbf{c}_k)\| > \varepsilon \) is bounded above by:
		\[
		\mathcal{O} \left( \dfrac{n^2}{\varepsilon^2} \right).
		\]
	\end{theorem}
	
	\begin{proof}
		We claim that the number of iterations satisfying \( \delta_k \geq \dfrac{\gamma \sqrt{n}\varepsilon}{L(\kappa^1_{n,\beta}) } \) is bounded above by $\mathcal{O} \left( {n^2/\varepsilon^2} \right).$ Notice that $\delta_k$ is non--increasing as $k$ increases. Suppose that all iterations up to $k=\bar{N}$ satisfy \( \delta_k \geq\dfrac{\gamma \sqrt{n}\varepsilon}{L(\kappa^1_{n,\beta}) } \) and \( \delta_{\bar{N}+1} < \dfrac{\gamma \sqrt{n}\varepsilon}{L(\kappa^1_{n,\beta}) } \).
		Analogous to Lemma~\ref{lem:min-decrease-reflection}, for every reflection step $k<\bar{N}$, we have:
		\[
		S^{(k+1)} - S^{(k)} \leq -(n+1)\dfrac{2\beta \gamma^2 \varepsilon^2}{L(\kappa^1_{n,\beta})^2}.
		\] 
		Suppose \(\bar{N}_r\) is the number of reflection steps up to \(\bar{N}\), with the corresponding number of shrinking steps \(\bar{N}_s\). Then
		\[
		\bar{N}_r  (n+1)\dfrac{2\beta \gamma^2 \varepsilon^2}{L(\kappa^1_{n,\beta})^2 } \leq (n+1) (d^{(0)}+\psi_0),
		\]
		which implies
		\[
		\bar{N}_r \leq \dfrac{L( d^{(0)}+\psi_0)(\kappa^1_{n,\beta})^2 }{2\beta \gamma^2 \varepsilon^2}.
		\]
		Notice that \(\bar{N}_s<\dfrac{\log(\bar{\delta}/\delta_0)}{\log \gamma}\), we conclude
		\[
		\bar{N} \leq \dfrac{L( d^{(0)}+\psi_0) {(\kappa^1_{n,\beta})^2} }{2\beta \gamma^2 \varepsilon^2}+\dfrac{\log(\bar{\delta}/\delta_0)}{\log \gamma}.
		\]
		
		Our claim holds, thus without loss of generality we assume iteration started from $\delta_0 < \dfrac{\gamma \sqrt{n}\varepsilon}{L(\kappa^1_{n,\beta}) }$.  It follows that the number of iterations satisfying \( \|\nabla f(\mathbf{c}_k)\| > \sqrt{n} \varepsilon \) for all $k$ is bounded above by: $\mathcal{O} \left( {n^2/\varepsilon^2} \right).$

		Suppose that the first iteration that $\|\nabla f(\mathbf{c}_k)\|$  dropped below $\sqrt{n}$ is indexed by \( N^1 + 1\), and let \( N^2 \) denote the remaining number of iterations until. \(N^2_r\) and \(N^2_s\) still represent corresponding reflection steps and shrinking steps.

		By \eqref{eq:PL}, we have
		\[
		f(\mathbf{c}_{N^1+1}) - f^\star \leq \dfrac{n \varepsilon^2}{2\mu}.
		\]
		From the interpolation error bound \eqref{eq:centroid-err}, for all $k$ we obtain
		\begin{align*}
			\left| f(\mathbf{c}_k) - \hat{f}(\mathbf{c}_k) \right|
			&= \left| f(\mathbf{c}_k) - \dfrac{1}{n+1} \sum_{i=1}^{n+1} f(\mathbf{x}_i^{(k)}) \right| \leq \dfrac{L\delta_k^2}{2} 
			\leq \dfrac{\gamma^2 n \varepsilon^2}{2L(\kappa^1_{n,\beta})^2}.
		\end{align*}
		The ascent due to all remaining shrinking steps is at most $(n+1)\psi_1$, where
		\[
		\psi_1 < \dfrac{L \gamma}{1 + \gamma} \delta_{N^1+1}^2< \dfrac{\gamma^3 n \varepsilon^2}{(1+\gamma)L(\kappa^1_{n,\beta})^2}.
		\]
		Consequently,
		\begin{align*}
			S^{(N^1+1)} - (n+1) f^\star
			&= (n+1)\!\left(\dfrac{1}{n+1} \sum_{i=1}^{n+1} f(\mathbf{x}_i^{(N^1+1)})-f^\star\right)\\
			&\le (n+1) \left(\left| f(\mathbf{c}_{N^1+1}) - \dfrac{1}{n+1} \sum_{i=1}^{n+1} f(\mathbf{x}_i^{(N^1+1)}) \right| + \left| f(\mathbf{c}_{N^1+1}) - f^\star \right| \right)\\
			&<(n+1) \left(\dfrac{\gamma^2 n \varepsilon^2}{2L(\kappa^1_{n,\beta})^2}+\dfrac{n\varepsilon^2}{2\mu}\right).
		\end{align*}
		By Lemma~\ref{lem:min-decrease-reflection}, we obtain 
		\begin{equation}\label{eq:K-sum-bound}
			S^{(N^1+1)}-(n+1)f^\star +(n+1)\psi_1\geq N^2_r (n+1)\dfrac{2\beta \gamma^2 \varepsilon^2}{L(\kappa^1_{n,\beta})^2 n }.
		\end{equation}
		It follows that
		\begin{equation}\label{eq:V-bound}
			N^2_r \leq \dfrac{L n^2}{4\beta^2} + \dfrac{L(\kappa^1_{n,\beta})^2n^2}{4\gamma^2\mu\beta}+ \dfrac{\gamma n^2}{2(1+\gamma)\beta} \in \mathcal{O}(n^4).
		\end{equation}
		Moreover, $N^2_s \leq \mathcal{O}(\log(\delta_{N^1}/\bar\delta))$, we conclude that when \( \varepsilon < {1}/{n} \), 
		\[
		N^2 \in\mathcal{O} \left( \dfrac{n^2}{\varepsilon^2} \right),
		\]
		which implies the stated worst--case complexity under the PL condition.
	\end{proof}
	
	\begin{remark}
		Actually, the complexity is $\mathcal{O}\left({n^2}/{\varepsilon^2}+n^4\right)$, where the first term is contributed by the phase with $\|\nabla f(\mathbf{c}_k)\| > \sqrt{n} \varepsilon$ and the final phase contributes at most an $\varepsilon$--independent  $\mathcal{O}(n^4)$ term (we omit the $\log$ term from shrinking here). The overall complexity will be $\mathcal{O} \left( {n^2}/{\varepsilon^2} \right)$ if $\varepsilon < {1}/{n}$.
	\end{remark}
	
	\begin{remark} \label{rm:7}
		If $f(\mathbf{x})$ satisfies the PL condition, we can establish the complexity of Algorithm~\ref{alg2} in terms of function value convergence. Specifically, let $\mathbf{x}_{N_\varepsilon}$ be the first iterate satisfying  $f(\mathbf{x}_{N_\varepsilon}) - f^\star \leq \varepsilon$, then the number of iterations satisfies $N_\varepsilon \in \mathcal{O}\left(n^3/{\varepsilon}\right)$. Using the same trick leads to $\mathcal{O}\left({n^2}/{\varepsilon}\right)$ when $\varepsilon <{1}/{n}$.
	\end{remark}
	
	\section{Complexity on Convex and Strongly Convex Smooth Functions}\label{sec:convex}
	For convex and strongly convex smooth functions, we analyze the complexity of Algorithm~\ref{alg2} in finding an $\varepsilon$--optimal solution.
	
	Let $N_\varepsilon$ be the first iteration such that the algorithm stops. Akin to Remark \ref{rm:7},
	\[  N_\varepsilon = \arg\min_k \left\{k \in \mathbb{N} ~\middle|~ \bar{d}^{(k)} \leq \varepsilon \right\}. 
	\] 
		
	\begin{assumption}\label{as3}
		In addition to Assumption~\ref{as1}, we assume that \( f(\mathbf{x}) \) is convex. Namely, for all $\mathbf{x}$ and $\mathbf{y} \in \R^n$, 
		\begin{equation}
			f(\mathbf{y}) \geq f(\mathbf{x}) + \langle \nabla f(\mathbf{x}), \mathbf{y} - \mathbf{x} \rangle.
		\end{equation}
		We also assume the sublevel set 
	\[
	\mathcal{L}_{\mathrm{avg},0}:=\Bigl\{\mathbf{x}\in\mathbb{R}^n:\ f(\mathbf{x})\le \dfrac{1}{n+1}\sum_{i=1}^{n+1}f(\mathbf{x}_i^{(0)})\Bigr\}.
	\]
		has a finite radius $R$, i.e., $\|\mathbf{x} - \mathbf{x}^\star\| \leq R$ for all $\mathbf{x} \in \mathcal{L}_{\mathrm{avg},0}$.
	\end{assumption}

	\subsection{Worst--Case Complexity under Convexity Assumption}
	We analyze $\bar{d}^{(k)}$ hereafter. Notice that convexity implies that $d^{(k)} \leq \bar{d}^{(k)}$ for all $k$. We also assume $\beta \geq \dfrac{1}{2}$ for simplicity. We first show that by convexity, shrinking steps are non--increasing for $d^{(k)}$.
	
	\begin{lemma}[Shrinking is Non--Increasing]\label{lem:convex-shrink-noninc}
		If iteration $k$ fails and performs a shrinking step with
		$\mathbf{x}_i^{(k+1)}=\gamma \mathbf{x}_i^{(k)}+(1-\gamma)\mathbf{x}_1^{(k)}$ for $i=2,\dots,n+1$, then
		\begin{equation}\label{eq:conv-shrink-noninc}
			S^{(k+1)}-S^{(k)}\ \le\ 0,
			\qquad\text{and equivalently}\qquad
			\bar d^{(k+1)}\ \le\ \bar d^{(k)}.
		\end{equation}
	\end{lemma}
	
	\begin{proof}
		For $i=2,\dots,n+1$, convexity gives
		$f(\mathbf{x}_i^{(k+1)}) \le \gamma f(\mathbf{x}_i^{(k)})+(1-\gamma) f(\mathbf{x}_1^{(k)})$.
		Summing this inequality over $i=2,\dots,n+1$ we obtain
		\[
			S^{(k+1)} \leq \gamma S^{(k)} + (n+1)(1-\gamma)\mathbf{x}_1^{(k)} \leq S^{(k)}.
		\]
	\end{proof}
	
	We can also rewrite the sufficient decrease condition \eqref{eq:suff-decrease} for simplicity; when iteration $k$ is successful,
	\begin{equation}\label{eq:conv-ref-drop}
		S^{(k+1)}-S^{(k)}\ \le\ -(n+1)\,\dfrac{2\beta}{n}\,L\,\delta_k^2
		\qquad\Longleftrightarrow\qquad
		\bar d^{(k+1)}\ \le\ \bar d^{(k)}\ -\ \dfrac{2\beta}{n}\,L\,\delta_k^2.
	\end{equation}
	
	\subsubsection{Gradient Lower Bound}
	Rewrite \eqref{eq:centroid-err-cvx} as:
	\begin{equation}\label{eq:conv-centroid-mean}
		d^{(k)}\ \ge\ \bar d^{(k)}-\dfrac{L}{2}\,\delta_k^2.
	\end{equation}
	
	We can demonstrate a lower bound for $\|\nabla f(\mathbf{c}_k)\|$ using the radius $R$ of the sublevel set as in Assumption~\ref{as3}.
	\begin{lemma}[Gradient Lower Bound with Centroid--Mean Correction]\label{lem:grad-lb-correct}
		Under Assumption~\ref{as3}, for all $k$,
		\begin{equation}\label{eq:conv-gradlb}
			\|\nabla f(\mathbf{c}_k)\|\ \ge\ \dfrac{d^{(k)}}{R}\ \ge\ \dfrac{\bar d^{(k)}-\dfrac{L}{2}\delta_k^2}{R}.
		\end{equation}
	\end{lemma}
	
	\begin{proof}
		Convexity gives $d^{(k)}\le \|\nabla f(\mathbf{c}_k)\|\,\|\mathbf{c}_k-\mathbf{x}^\star\|$ for  $\mathbf{x}^\star\in\arg\min f$. 
		
		Notice that $f(\mathbf{c}_k)\le \frac{1}{n+1}\sum_{i=1}^{n+1}f(\mathbf{x}_i^{(k)})\le \frac{1}{n+1}\sum_{i=1}^{n+1}f(\mathbf{x}_i^{(0)})$, thus $\mathbf{c}_k \in \mathcal{L}_{\mathrm{avg},0}$.   
		
		Take $\|\mathbf{c}_k-\mathbf{x}^\star\|\le R$ to get the first inequality; combine with \eqref{eq:conv-centroid-mean} for the second.
	\end{proof}
	
	\subsubsection{Sufficient Condition for Reflection (Convex Case)}
	
	\begin{lemma}[Gradient Upper Bound at a Shrinking Step]\label{lem:grad-rej}
		If the reflection at iteration $k$ is rejected (i.e., \eqref{eq:suff-decrease} fails), then the gradient at the centroid is bounded by:
		\begin{equation}\label{eq:up-bound-convex}
			\|\nabla f(c_k)\| \leq \kappa^2_{n,\beta} L \delta_k.
		\end{equation}
		where $\kappa^2_{n,\beta}:= (\beta-\dfrac{1}{2})n + \dfrac{\sqrt{n}}{2}-\dfrac{1}{2}$.
	\end{lemma}
	
	\begin{proof}
		Akin to Lemma~\ref{lem:sufficient-reflect}, rejection implies that 
		\begin{align*}
			v_r^{(k)} - v^{(k)} &\geq -\dfrac{2n+2}{n}\beta L \delta_k^2, \\
			\left| v_r^{(k)} + v^{(k)} \right| &\leq \left(1+\dfrac{1}{n}\right)^2 L \delta_k^2, \\
			\dfrac{n}{n+1}v^{(k)} &\geq \dfrac{\delta_k}{n} \left( \|\nabla f(\mathbf{c}_k)\| - \dfrac{\sqrt{n}}{2} L \delta_k \right).
		\end{align*}
		
		Performing an almost identical derivation, we can obtain \eqref{eq:up-bound-convex}.
	\end{proof}
	
	From the lower bound of $\|\nabla f(\mathbf{c}_k)\|$ for all $k$ and upper bound of $\|\nabla f(\mathbf{c}_k)\|$ for shrinking steps, we can choose $\delta_k$ small enough than $\bar{d}^{(k)}$ such that the lower bound is greater than the upper bound. The contradiction shows that such $k$ will always be a reflection step. Such description about \emph{how small} $\delta_k$ is compared with $\bar{d}^{(k)}$ is defined below as the \emph{adaptive tail radius}.
	
	\begin{definition}
		Fix $\eta\in(0,1)$, the \emph{adaptive tail radius} is
			\begin{equation}\label{eq:conv-delta-tail}
			\delta_{\rm tail}(\bar d)\ :=\ \min\!\left\{\,\dfrac{(1-\eta)\,\bar d}{2\,\kappa^2_{n,\beta}\,L\,R}\ ,\ \sqrt{\dfrac{(1-\eta)\,\bar d}{L}}\,\right\}.
		\end{equation}
	\end{definition}

	Set the \emph{branch threshold}
	\begin{equation}\label{eq:conv-dthr}
		\bar d_{\rm thr}\ :=\ \dfrac{4\,(\kappa^2_{n,\beta})^2\,L\,R^2}{1-\eta}.
	\end{equation}
	Then, $\delta_{\rm tail}(\bar d)=\dfrac{(1-\eta)\bar d}{2\kappa^2_{n,\beta}LR}$ when $\bar d\le \bar d_{\rm thr}$ (first branch),
	and $\delta_{\rm tail}(\bar d)=\sqrt{\dfrac{(1-\eta)\bar d}{L}}$ when $\bar d>\bar d_{\rm thr}$ (second branch).
	
	\begin{lemma}[Sufficient Condition for Reflection]\label{lem:conv-tail-accept}
		If $\delta_k\le \delta_{\rm tail}(\bar d^{(k)})$, then iteration $k$ \emph{cannot} be rejected and must be a reflection step.
	\end{lemma}
	
	\begin{proof}
		If rejected, \eqref{eq:up-bound-convex} gives:
		\[
		\|\nabla f(\mathbf{c}_k)\|\le \kappa^2_{n,\beta}L\,\delta_k\le \kappa^2_{n,\beta}L\,\delta_{\rm tail}(\bar d^{(k)}).
		\]
		
		By Lemma~\ref{lem:grad-lb-correct} and \eqref{eq:conv-delta-tail},
		\[
		\|\nabla f(\mathbf{c}_k)\|\ \ge\ \dfrac{\bar d^{(k)}-\dfrac{L}{2}\delta_k^2}{R}
		\ \ge\ \dfrac{\bar d^{(k)}}{R}-\dfrac{L}{2R}\,\delta_{\rm tail}(\bar d^{(k)})^2
		\ \ge\ \dfrac{1+\eta}{2}\cdot\dfrac{\bar d^{(k)}}{R}.
		\]
		Nonetheless $\kappa^2_{n,\beta}L\,\delta_{\rm tail}(\bar d^{(k)})\le \dfrac{1-\eta}{2}\cdot\dfrac{\bar d^{(k)}}{R}$ by the first branch in \eqref{eq:conv-delta-tail}, and thus we have a contradiction.
	\end{proof}
	
	\subsubsection{Radius Lower Bound}
	
	We define the \emph{tail region} as iterations after the condition $\delta_{k} \leq \delta_{\rm tail}(\bar d^{(k)})$ is satisfied \emph{for the first time}. The following lemma, akin to \eqref{eq:ncvglb}, shows that the simplex radius in the tail region will not be too small and thus have a lower bound.
	
	\begin{lemma}[Radius Lower Bound after Entry the Tail Region]\label{lem:conv-radius-lb} Relation
		\begin{equation}\label{eq:conv-radiuslb}
			\delta_k\ \ge\ \gamma\,\delta_{\rm tail}(\bar d^{(k)})
		\end{equation}
		
		holds for all $k$ if at some $t\leq k$, $\delta_{t} \leq \delta_{\rm tail}(\bar d^{(t)}).$
	\end{lemma}
	\begin{proof}
		Let $k_{\rm in}$ be the first index with $\delta_{k_{\rm in}}\le \delta_{\rm tail}(\bar d^{(k_{\rm in})})$.
		
		The preceding step $k_{\rm in-1}$ must be a shrinking step at radius larger than $\delta_{\rm tail}(\bar d^{(k_{\rm in})})$. Hence
		\[
		\delta_{k_{\rm in}}=\gamma\,\delta_{k_{\rm in}-1}\ >\ \gamma\,\delta_{\rm tail}(\bar d^{(k_{\rm in})}).
		\]
		
		Since accepted reflections keep $\delta$ unchanged and $\bar d$ is non--increasing by Lemma~\ref{lem:convex-shrink-noninc}, all following reflection steps $k>k_{\rm in}$ before next shrinking step satisfy $\delta_k\ \ge\ \gamma\,\delta_{\rm tail}(\bar d^{(k)})$.
		
		At a new shrinking step ${k_\mathrm{shr}}$, Lemma \ref{lem:conv-tail-accept} shows that $\delta_{k_\mathrm{shr}}\geq \delta_{\rm tail}(\bar d^{({k_\mathrm{shr}})})$, thus $\delta_{k_{\mathrm{shr+1}}} = \gamma \delta_{k_\mathrm{shr}}\geq \gamma \delta_{\rm tail}(\bar d^{({k_\mathrm{shr}})}) \geq \gamma \delta_{\rm tail}(\bar d^{\left({k_\mathrm{shr+1}}\right)})$. Similarly all iterations in the tail region satisfy \eqref{eq:conv-radiuslb}.

	\end{proof}
	
	\begin{corollary}[Global Radius Lower Bound]\label{cro:lb}
		Relation 
		\begin{equation*}
			\delta_k\ \ge\ \gamma\,\delta_{\rm tail}(\bar d^{(k)})
		\end{equation*}
		
		holds for all $k$ if $\delta_0>\delta_{\rm tail}(\bar d^{(0)})$ holds.
	\end{corollary}	
	\begin{proof}
		If at some $t$, $\delta_{t} \leq \delta_{\rm tail}(\bar d^{(t)})$ for the first time, as in Lemma~\ref{lem:conv-radius-lb} the relation would hold for all $k \geq t$.
		
		Note that  for all $k < t$ , $\delta_k\ > \delta_{\rm tail}(\bar d^{(k)}) > \gamma\,\delta_{\rm tail}(\bar d^{(k)})$, which finishes the proof.
	\end{proof}
	
	\subsubsection{Two--Phase of Decrease}
	Define
	\begin{equation}\label{eq:conv-C1A}
		C_1 := \dfrac{2\beta}{n}\gamma^2(1-\eta)\in(0,1),\qquad
		A := \dfrac{\beta\gamma^2(1-\eta)^2}{2L\,R^2 n(\kappa^2_{n,\beta})^2}.
	\end{equation}
	
	Let $k_{\rm sw} := \min\{k\ge 0:\ \bar d^{(k)}\le \bar d_{\rm thr}\}$. We define the iteration index before switching (Phase \Rmnum{1}), i.e. $\bar d^{(k)}>\bar d_{\rm thr}$ as $k \in [0,k_{\rm sw})$. The index after switching (Phase \Rmnum{2}) is when $\bar d^{(k)}\leq \bar d_{\rm thr}$.
	
	\begin{lemma}[Phase \Rmnum{1}: Linear Decrease]\label{lem:phase1}
		If $\delta_0>\delta_{\rm tail}(\bar d^{(0)})$ and at successful iteration $k$, $\bar d^{(k)}>\bar d_{\rm thr}$, then $\delta_{\rm tail}(\bar d^{(k)})=\sqrt{\dfrac{(1-\eta)\bar d^{(k)}}{L}}$ and
		\[
		\bar d^{(k+1)}\ \le\ \bar d^{(k)}\ -\ \dfrac{2\beta}{n}\,L\,\delta_k^2
		\ \le\ (1-C_1)\,\bar d^{(k)}.
		\]
		Consequently, the successful iteration number in this Phase is:
		\begin{equation}\label{eq:conv-Nphase1}
			N_{\rm Phase\,I}\ \le\  \dfrac{\log\!\big(\bar d^{(0)}/\bar d_{\rm thr}\big)}{\log\!\big(1/(1-C_1)\big)} 
			\ = \mathcal{O}\!\left(\dfrac{n}{\beta\,\gamma^2}\,\log\dfrac{\bar d^{(0)}}{\bar d_{\rm thr}}\right).
		\end{equation}
	\end{lemma}
	
	\begin{proof}
		By the branch threshold definition, the first relation is obvious.
		By the radius lower bound Corollary~\ref{cro:lb},
		\begin{equation}\label{eq:pf1-radiuslb}
			\delta_k\ \ge\ \gamma\,\delta_{\rm tail}(\bar d^{(k)})\ =\ \gamma\,\sqrt{\dfrac{(1-\eta)\,\bar d^{(k)}}{L}}.
		\end{equation}
		Applying the accepted--reflection average decrease \eqref{eq:conv-ref-drop}
		\[
		\bar d^{(k+1)}\ \le\ \bar d^{(k)}\ -\ \dfrac{2\beta}{n}\,L\,\delta_k^2
		\ \stackrel{\eqref{eq:pf1-radiuslb}}{\le}\ \bar d^{(k)}\ -\ \dfrac{2\beta}{n}\,L\,\gamma^2\cdot \dfrac{(1-\eta)}{L}\,\bar d^{(k)}
		\ =\ \bigl(1-C_1\bigr)\,\bar d^{(k)},
		\]
		where $C_1:=\dfrac{2\beta}{n}\,\gamma^2(1-\eta)\in(0,1)$.
		
		Iterating this linear decrease over the accepted--reflection subsequence on the interval $[0,k_{\rm sw})$ yields
		\[
		\bar d^{(N_{\rm Phase\,I})}\ \le\ (1-C_1)^{N_{\rm Phase\,I}}\,\bar d^{(0)}.
		\]
		The upper bound for $N_{\rm Phase\,I}$ with $\bar d^{(N_{\rm Phase\,I})}\le \bar d_{\rm thr}$ is
		\[
		N_{\rm Phase\,I}\ \le\ \dfrac{\log\!\big(\bar d^{(0)}/\bar d_{\rm thr}\big)}{\log\!\big(1/(1-C_1)\big)}.
		\]
		Using $\log\!\big(1/(1-C_1)\big)\ge C_1$ for $C_1\in(0,1)$ yields the desired
		$\mathcal{O}\big(\dfrac{n}{\beta\gamma^2}\log(\bar d^{(0)}/\bar d_{\rm thr})\big)$ bound.
	\end{proof}
	
	\begin{lemma}[Phase \Rmnum{2}: Sublinear Decrease]\label{lem:phase2}
		If $\delta_0>\delta_{\rm tail}(\bar d^{(0)})$ and at successful iteration $k$, $\bar d^{(k)}\le \bar d_{\rm thr}$, then $\delta_{\rm tail}(\bar d^{(k)})=\dfrac{(1-\eta)\bar d^{(k)}}{2\kappa^2_{n,\beta}LR}$ and
		\[
		\bar d^{(k+1)}\ \le\ \bar d^{(k)}\ -\ \dfrac{2\beta}{n}\,L\,\delta_k^2
		\ \le\ \bar d^{(k)}\ -\ A\big(\bar d^{(k)}\big)^2.
		\]
		where $A := \dfrac{\beta\gamma^2(1-\eta)^2}{2L\,R^2 n(\kappa^2_{n,\beta})^2}$.
		
		In addition, to reach $\bar d^{(k)}\le \varepsilon$ from $\bar d_{\rm thr}$, the successful iteration number in this phase is:
		\begin{equation}\label{eq:conv-Nphase2}
			N_{\rm Phase\,II}\ \le\ \dfrac{1}{A}\left(\dfrac{1}{\varepsilon}-\dfrac{1}{\bar d_{\rm thr}}\right).
		\end{equation}
	\end{lemma}
	
	\begin{proof}
		Akin to Lemma~\ref{lem:phase1},
		\begin{equation}\label{eq:pf2-radiuslb}
			\delta_k\ \ge\ \gamma\,\delta_{\rm tail}(\bar d^{(k)})\ =\ \gamma\,\dfrac{(1-\eta)\,\bar d^{(k)}}{2\,\kappa^2_{n,\beta}L R}.
		\end{equation}
		Apply \eqref{eq:conv-ref-drop} and \eqref{eq:pf2-radiuslb}:
		\[
		\bar d^{(k+1)}\ \le\ \bar d^{(k)}\ -\ \dfrac{2\beta}{n}\,L\,\delta_k^2
		\ \le\ \bar d^{(k)}\ -\ \dfrac{2\beta}{n}\,L\,\gamma^2\left(\dfrac{(1-\eta)\,\bar d^{(k)}}{2\,\kappa^2_{n,\beta}L R}\right)^{\!2}
		\ =\ \bar d^{(k)}\ -\ A\,\big(\bar d^{(k)}\big)^2,
		\]
		with $A=\dfrac{\beta\,\gamma^2(1-\eta)^2}{2\,n\,(\kappa^2_{n,\beta})^2\,L\,R^2}$.
		
		Rearrange to the \emph{inverse--gap} recursion:
		\[
		\dfrac{1}{\bar d^{(k+1)}}\ \ge\ \dfrac{1}{\bar d^{(k)}}\ +\ A.
		\]
		Summing over $N_{\rm Phase\,II}$ accepted reflections from $k$ inside Phase~II yields 
		\[
		\dfrac{1}{\bar d^{(k+N_{\rm Phase\,II})}}\ \ge\ \dfrac{1}{\bar d^{(k)}}\ +\ N_{\rm Phase\,II} A.
		\]
		The smallest $N_{\rm Phase\,II}$ such that $\bar d^{(k+N_{\rm Phase\,II})}\le \varepsilon$ therefore satisfies
		\[
		N_{\rm Phase\,II}\ \le\ \dfrac{1}{A}\left(\dfrac{1}{\varepsilon}-\dfrac{1}{\bar d_{\rm thr}}\right).
		\]
	\end{proof}
	
	\subsection{Worst--Case Complexity (\Rmnum{3})}
	\begin{lemma}[Number of Shrinking Steps under the Convex Stopping Rule]\label{lem:num-shrink-steps-conv}
		Suppose that Assumption~\ref{as3} holds,  $\delta_0>\bar{\delta}_{\rm cvx}$ and the stopping condition is not satisfied yet, where
		\begin{equation}\label{eq:def-delta-cvx}
			\bar{\delta}_{\rm cvx}\ :=\ \min\!\left\{\ \dfrac{(1-\eta)\,\varepsilon}{2\,\kappa^2_{n,\beta}\,L\,R}\ ,\ \sqrt{\dfrac{(1-\eta)\,\varepsilon}{L}}\ \right\}.
		\end{equation}
		Denote by \(N_r\) and \(N_s\) the numbers of reflection and shrinking steps among the \(N_\varepsilon\) iterations, respectively.
		Then
		\begin{equation}\label{eq:num-shrink-steps-conv}
			N_s\ <\ \dfrac{\log\!\big(\bar{\delta}_{\rm cvx}/\delta_0\big)}{\log \gamma}
			\ =\ \dfrac{\log\!\big(\delta_0/\bar{\delta}_{\rm cvx}\big)}{\log (1/\gamma)}.
		\end{equation}
	\end{lemma}
	
	\begin{proof}
		By definition of \(N_\varepsilon\), we have \(\bar d^{(k)}\ge \varepsilon\) for all \(k\le N_\varepsilon-1\).
		Fix any \(k\le N_\varepsilon-1\) that is a shrinking step (i.e., the reflection is rejected and the simplex is contracted).
		
		At such \(k\), we must have \(\delta_k>\bar{\delta}_{\rm cvx}\); otherwise, if \(\delta_k\le \bar{\delta}_{\rm cvx}\), then
		\(\delta_k\le \delta_{\rm tail}(\bar d^{(k)})\) because \(\bar d^{(k)}\ge \varepsilon\) and \(\delta_{\rm tail}(\cdot)\) is increasing,
		and by the Lemma~\ref{lem:conv-tail-accept} the reflection cannot be rejected, which is a contradiction.
		
		Each shrinking multiplies the radius by \(\gamma\in(0,1)\).
		
		After \(N_s\) shrinkings,
		\(\delta_0\,\gamma^{N_s}>\bar{\delta}_{\rm cvx}\).
		Taking logs (note \(\log\gamma<0\)) gives
		\[
		\log\gamma\cdot N_s\ <\ \log\!\big(\bar{\delta}_{\rm cvx}/\delta_0\big),
		\]
		which is exactly \eqref{eq:num-shrink-steps-conv}.
	\end{proof}

	\begin{theorem}[Worst--Case Complexity in Convex Case]\label{thm:convex-wc}
		Suppose Assumption~\ref{as3} holds and $\delta_0>\delta_{\rm tail}(\bar d^{(0)})$. Then for Algorithm~\ref{alg2} the number of iterations needed to ensure
		$\bar d^{(k)}\le \varepsilon$ is bounded by:
		\begin{equation}\label{eq:conv-wc-bound}
			N_{\varepsilon}\ \le\
			\dfrac{\log\!\big(\bar d^{(0)}/\bar d_{\rm thr}\big)}{\log\!\big(1/(1-C_1)\big)} + \dfrac{1}{A}\left(\dfrac{1}{\varepsilon}-\dfrac{1}{\bar d_{\rm thr}}\right) +
			\dfrac{\log\!\big(\delta_0/\bar{\delta}_{\rm cvx}\big)}{\log (1/\gamma)},
		\end{equation}
		where $C_1,A$ are given in \eqref{eq:conv-C1A}, $\bar d_{\rm thr}$ in \eqref{eq:conv-dthr}, and $\bar{\delta}_{\rm cvx}$ in \eqref{eq:def-delta-cvx}.
		
		In particular, if we choose $\beta = \dfrac{1}{2}$, 
		\begin{equation}\label{eq:conv-wc-order}
			N_{\varepsilon}\ =\mathcal{O}\left(\dfrac{n^2}{\varepsilon} + \log\dfrac{1}{\varepsilon} + n\log n\right).
		\end{equation}
		
		Since $d^{(k)}\le \bar d^{(k)}$ for all $k$, the same bound certifies $d^{(k)}\le \varepsilon$.
	\end{theorem}
	
	\begin{proof}
		Combining Lemma~\ref{lem:phase1}, Lemma~\ref{lem:phase2}, and Lemma~\ref{lem:num-shrink-steps-conv} to get \eqref{eq:conv-wc-bound}. Notice that $C_1 \in \mathcal{O}\left({1}/{n}\right)$, $A \in \mathcal{O}\left({1}/{n^2}\right)$ and $\bar{\delta}_{\rm thr} \in \mathcal{O}\left(n^2\right)$ when $\dfrac{1}{2}$, together lead to \eqref{eq:conv-wc-order}.
	\end{proof}
	
	\subsection{Worst--Case Complexity under Strong Convexity Assumptions}

	\begin{assumption}\label{as4}
		In addition to Assumption~\ref{as3}, we further assume that \( f(\mathbf{x}) \) is $\mu$--strongly convex. Namely, there exists a constant \( \mu > 0 \) such that:
		\begin{equation}
			f(\mathbf{y})\ge f(\mathbf{x})+\langle \nabla f(\mathbf{x}),\mathbf{y}-\mathbf{x}\rangle+\dfrac{\mu}{2}\|\mathbf{y}-\mathbf{x}\|^2, \quad \forall \mathbf{x},\mathbf{y} \in \R^n.
		\end{equation}
	\end{assumption}
	
	Notice that strong convexity implies the PL condition with parameter $\mu$. We show that RSSM enjoys \emph{linear convergence}, hence a $\log(1/\varepsilon)$ iteration bound.
	
	\begin{theorem}[Linear convergence under strong convexity]\label{thm:sc-wc}
		Suppose Assumption~\ref{as4} holds. Consider any index $j$ at which the reflection is rejected (so a shrinking step occurs). Then for every subsequent \emph{accepted reflection} $k>j$ until the next shrinking occurs, the average gap $\bar d^{(k)}$ decreases linearly as:
		\begin{equation}\label{eq:sc-linear-conv}
			\bar d^{(k+1)} \;\le\; (1-\rho)\,\bar d^{(k)} ,
		\end{equation}
		with contraction factor
		\begin{equation}\label{eq:sc-rho}
			\rho := \dfrac{4\,\beta\,\mu\,\gamma^2}{\,n\bigl(L(\kappa^2_{n,\beta})^2+\mu\bigr)}\ > 0.  
		\end{equation}
		Consequently, the number of successful iterations to reach $\bar d^{(k)}\le \varepsilon$ is
		\[
		N_\varepsilon \;=\; \mathcal{O}\left(\dfrac{1}{\rho}\log\dfrac{1}{\varepsilon}\right).
		\]
	\end{theorem}
	
	\begin{proof}
		At any rejected index $j$, Lemma~\ref{lem:grad-rej} gives
		\[
		\|\nabla f(\mathbf{c}_j)\|\ \le\ \kappa^2_{n,\beta}\,L\,\delta_j.
		\]
		Strong convexity and Lemma~\ref{lem:grad-lb-correct} imply
		\[
		\|\nabla f(\mathbf{c}_j)\|^2 \;\ge\; 2\mu\,d^{(j)} \;\ge\; 2\mu\Bigl(\bar d^{(j)}-\dfrac{L}{2}\,\delta_j^2\Bigr).
		\]
		Combining the two yields
		\[
		(\kappa^2_{n,\beta})^2L^2\,\delta_j^2 \;\ge\; 2\mu\,\bar d^{(j)}-\mu L\,\delta_j^2
		\ \Longrightarrow\
		\delta_j^2 \;\ge\; \dfrac{2\mu\,\bar d^{(j)}}{\,L\bigl(L(\kappa^2_{n,\beta})^2+\mu\bigr)}.
		\]
		After shrinking, $\delta_{j+1}=\gamma\,\delta_j$. Every \emph{accepted reflection} $k$ in the subsequent block keeps the radius, so $\delta_k=\delta_{j+1}$, and hence
		\begin{equation}\label{eq:sc-radius-lb-proof}
			\delta_k^2 \;=\; \gamma^2\,\delta_j^2 \;\ge\; \dfrac{2\mu\,\gamma^2\,\bar d^{(j)}}{\,L\bigl(L(\kappa^2_{n,\beta})^2+\mu\bigr)}.
		\end{equation}
		
		For any accepted reflection, \eqref{eq:conv-ref-drop} gives
		\[
		\bar d^{(k+1)} \;\le\; \bar d^{(k)} - \dfrac{2\beta}{n}\,L\,\delta_k^2.
		\]
		Using \eqref{eq:sc-radius-lb-proof} we get
		\[
		\bar d^{(k+1)} \;\le\; \bar d^{(k)} - \dfrac{2\beta}{n}L\cdot \dfrac{2\mu\,\gamma^2\,\bar d^{(j)}}{\,L\bigl(L(\kappa^2_{n,\beta})^2+\mu\bigr)}
		\;=\; \bar d^{(k)} - \rho\,\bar d^{(j)}.
		\]
		Since $\bar d^{(\cdot)}$ is non--increasing, we have $\bar d^{(k)}\le \bar d^{(j)}$, and hence
		\[
		\bar d^{(k+1)} \;\le\; \bar d^{(k)} - \rho\,\bar d^{(k)} \;=\; (1-\rho)\,\bar d^{(k)}.
		\]
		This proves the linear contraction \eqref{eq:sc-linear-conv}. 
		
		Summing geometric decrease yields the stated $\mathcal{O}\left(\dfrac{1}{\rho}\log\dfrac{1}{\varepsilon}\right)$ bound. 
	\end{proof}
	
	\begin{remark}
		The shrinking steps still can be bounded by Lemma~\ref{lem:num-shrink-steps-conv}. Since $\rho \in \mathcal{O}(n^{-2})$ when $\beta = \dfrac{1}{2}$, the overall iteration complexity is $\mathcal{O}\!\left(n^2 \log ({1}/{\varepsilon}) + \log ({1}/{\varepsilon})\right)$.
	\end{remark}
	
	\begin{remark}
		Note that Theorem~ \ref{thm:sc-wc} relies on the occurrence of at least one shrinking step. If $\delta_0$ is chosen such that $\delta_0^2 \ge \dfrac{2\mu\bar d^{(0)}}{L\left(L(\kappa^2_{n,\beta})^2+\mu\right)}$, similar induction holds. We omit detailed explanation here for simplicity.
	\end{remark}
	
	\section{Conclusion}\label{sec:conclusion}
	
	This paper investigates the worst--case complexity of the regular simplicial search method (RSSM) that involves reflection and shrinking steps inspired by the classical method of Spendley et al.~\cite{Spe62}. Based on the sharp interpolation and extrapolaition error bound proposed by Cao et al.~\cite{Cao23}, we estimate the function value change in reflection and shrinking quantitatively. Under various assumptions, we establish explicit worst--case bounds on the number of iterations to reach an $\varepsilon$--stationary point in nonconvex settings and an $\varepsilon$--optimal solution in convex settings. To the best of our knowledge, these are the first complexity results for a simplex--type method with shrinking steps. This paper pioneers in establishing complexity results for simplex--type methods, laying a foundation for the analysis of more advanced variants in future work. 
	
	\newpage
	
	\appendix
	\section*{Appendix: Detailed Error Bounds Calculation}
	
	\subsection*{A.1. Setup and Basic Geometry}\label{ap:a1}
	
	We restate our setup here for convenience.
	
	Let $f:\R^n\to\R$ be $L$--smooth (i.e., $\|\nabla f(\mathbf{x})-\nabla f(\mathbf{y})\|\le L\|\mathbf{x}-\mathbf{y}\|$).
	At iteration $k$, the simplex $\Theta_k=\{\mathbf{x}^{(k)}_1,\ldots,\mathbf{x}^{(k)}_{n+1}\}$ is \emph{regular} with centroid
	\[
	\mathbf{c}_k=\dfrac{1}{n+1}\sum_{i=1}^{n+1}\mathbf{x}^{(k)}_i,\qquad \|\mathbf{x}^{(k)}_i-\mathbf{c}_k\|=\delta_k\ \ \text{for all }i.
	\]
	Write $\mathbf{v}_i:=\mathbf{x}^{(k)}_i-\mathbf{c}_k$. Then
	\[
	\sum_{i=1}^{n+1} \mathbf{v}_i=0,\qquad
	\mathbf{v}_i^\top \mathbf{v}_j=\begin{cases}\delta_k^2,& i=j,\\[2pt]-\dfrac{\delta_k^2}{n},& i\ne j,\end{cases}
	\qquad
	\sum_{i=1}^{n+1} \mathbf{v}_i \mathbf{v}_i^\top=\dfrac{n+1}{n}\,\delta_k^2\,I_n,
	\]
	and $\|\mathbf{x}^{(k)}_i-\mathbf{x}^{(k)}_j\|^2=2(1+\dfrac1n)\delta_k^2$ for $i\ne j$.
	
	We write $\mathbf{1}_n$ for the all--one vector in $\R^n$.
	\subsection*{A.2. Computing $G$ and $M$}\label{ap:a2}
	
	\subsubsection*{(i) Reflection Step}
	The reflection point is $\mathbf{x}^{(k)}_r=-\mathbf{x}^{(k)}_{n+1}+\dfrac{2}{n}\sum_{i=1}^n \mathbf{x}^{(k)}_i$.
	Use the coefficients
	\[
	\ell_0=-1\ \ (\mathbf{x}_0=\mathbf{x}^{(k)}_r),\qquad
	\ell_i=\dfrac{2}{n}\ (i=1,\ldots,n),\qquad
	\ell_{n+1}=-1 .
	\]
	Let $\bar{\mathbf{x}}=\dfrac{1}{n}\sum_{i=1}^n \mathbf{x}^{(k)}_i=\mathbf{c}_k-\dfrac{1}{n}\mathbf{v}_{n+1}$. A short calculation gives
	\[
	G=\dfrac{2}{n}\sum_{i=1}^n (\mathbf{x}^{(k)}_i-\bar{\mathbf{x}})(\mathbf{x}^{(k)}_i-\bar{\mathbf{x}})^\top
	-2(\bar{\mathbf{x}}-\mathbf{x}^{(k)}_{n+1})(\bar{\mathbf{x}}-\mathbf{x}^{(k)}_{n+1})^\top .
	\]
	Using the identities in A.1 and writing $\mathbf{u}:=\mathbf{v}_{n+1}/\|\mathbf{v}_{n+1}\|$, we obtain the spectral form
	\[
	G=\dfrac{2(n+1)}{n^2}\,\delta_k^2\,I_n
	-\dfrac{2(n+1)(n+2)}{n^2}\,\delta_k^2\,\mathbf{u}\mathbf{u}^\top .
	\]
	Hence, the eigenvalues of $G$ are
	\[
	\underbrace{\lambda_+=\dfrac{2(n+1)}{n^2}\,\delta_k^2}_{\text{mult. }n-1},\qquad
	\lambda_-=-\dfrac{2(n+1)^2}{n^2}\,\delta_k^2 .
	\]
	Therefore the nuclear norm is
	\[
	\|G\|_\ast=(n-1)\lambda_+ + |\lambda_-|= \dfrac{4(n+1)}{n}\,\delta_k^2.
	\]
	That is exactly:
	\[
	\boxed{\;
		\bigl| f(\mathbf{x}^{(k)}_r)+f(\mathbf{x}^{(k)}_{n+1})-\dfrac{2}{n}\textstyle\sum_{i=1}^n f(\mathbf{x}^{(k)}_i) \bigr|
		\;\le\; \dfrac{L}{2}\,\|G\|_\ast
		\;=\; \dfrac{2n+2}{n}\,L\,\delta_k^2 .\;
	}
	\]
	
	\subsubsection*{(ii) Centroid Query}
	Take $\mathbf{x}=\mathbf{c}_k$ and $\ell_i=\dfrac{1}{n+1}$ for $i=1,\ldots,n+1$, $\ell_0=-1$.
	Then
	\[
	G=\dfrac{1}{n+1}\sum_{i=1}^{n+1} \mathbf{x}^{(k)}_i \bigl(\mathbf{x}^{(k)}_i\bigr)^\top - \mathbf{c}_k \mathbf{c}_k^\top
	=\dfrac{1}{n+1}\sum_{i=1}^{n+1} \mathbf{v}_i \mathbf{v}_i^\top
	=\dfrac{\delta_k^2}{n}\,I_n .
	\]
	Thus $\|G\|_\ast=\delta_k^2$, so
	\[
	\boxed{\;
		\biggl| f(\mathbf{c}_k)-\dfrac{1}{n+1}\sum_{i=1}^{n+1} f(\mathbf{x}^{(k)}_i) \biggr|
		\;\le\; \dfrac{L}{2}\,\delta_k^2 .\;
	}
	\]
	
	\subsubsection*{(iii) Shrinking Step}
	Fix $i\in\{2,\ldots,n+1\}$ and set $\hat{\mathbf{x}}_i=\gamma \mathbf{x}^{(k)}_i+(1-\gamma)\mathbf{x}^{(k)}_1$ with $\gamma \in (0,1)$.
	
	Use the two-point coefficients:
	\[
	\ell_0=-1\ \ (\mathbf{x}_0=\hat{\mathbf{x}}_i),\qquad
	\ell_1=1-\gamma,\qquad
	\ell_i=\gamma,\qquad
	\ell_j=0\ (j\notin\{1,i\}) .
	\]
	A direct calculation yields
	\[
	G=\gamma(1-\gamma)\,\bigl(\mathbf{x}^{(k)}_i-\mathbf{x}^{(k)}_1\bigr)\bigl(\mathbf{x}^{(k)}_i-\mathbf{x}^{(k)}_1\bigr)^\top ,
	\]
	so $G\succeq0$ and $\|G\|_\ast=\tr(G)=\gamma(1-\gamma)\|\mathbf{x}^{(k)}_i-\mathbf{x}^{(k)}_1\|^2$.
	Using $\|\mathbf{x}^{(k)}_i-\mathbf{x}^{(k)}_1\|^2=2(1+\dfrac1n)\delta_k^2$, we obtain:
	\[
	\boxed{\;
		\bigl| f(\hat{\mathbf{x}}_i)-\gamma f(\mathbf{x}^{(k)}_i)-(1-\gamma)f(\mathbf{x}^{(k)}_1) \bigr|
		\;\le\; \dfrac{L}{2}\,\|G\|_\ast
		\;=\; \dfrac{n+1}{n}\,L\,\gamma(1-\gamma)\,\delta_k^2 .\;
	}
	\]
	
    \subsection*{\hypertarget{app:reflection}{A.3. Reflection Step}}  
	\subsubsection*{(i) Reflection Step}
	
	We keep the notation of A.1: $\mathbf{c}_k=\dfrac{1}{n+1}\sum_{i=1}^{n+1}\mathbf{x}^{(k)}_i$ and $\mathbf{v}_i:=\mathbf{x}^{(k)}_i-\mathbf{c}_k$, so that
	\[
	\sum_{i=1}^{n+1}\mathbf{v}_i=0,\qquad 
	\mathbf{v}_i^\top\mathbf{v}_j=
	\begin{cases}
		\delta_k^2,& i=j,\\[2pt]
		-\dfrac{\delta_k^2}{n},& i\neq j.
	\end{cases}
	\]
	The reflection query point is
	\[
	\mathbf{x}=\mathbf{x}_r^{(k)}\;=\;-\mathbf{x}^{(k)}_{n+1}+\dfrac{2}{n}\sum_{i=1}^n \mathbf{x}^{(k)}_i.
	\]
	Let $\mathbf{u}:=\mathbf{v}_{n+1}/\|\mathbf{v}_{n+1}\|=\mathbf{v}_{n+1}/\delta_k$. From A.2 we know that the negative eigenspace of $G$ is one--dimensional and aligned with $\mathbf{u}$, hence we may take $P_-=\mathbf{u}\in\mathbb{R}^{n\times 1}$.
	
	As in \cite{Cao23}, let $Y\in\mathbb{R}^{(n+1)\times n}$ have $i$th row $(\mathbf{x}^{(k)}_i-\mathbf{x}_r^{(k)})^\top$. For the reflection coefficients
	\[
	\ell_0=-1,\qquad \ell_i=\dfrac{2}{n}\ (i=1,\dots,n),\qquad \ell_{n+1}=-1,
	\]
	we have $\mathcal{I}_+=\{1,\dots,n\}$ and $\mathcal{I}_-=\{0,n+1\}$, so that $Y_+$ is formed by the first $n$ rows and $Y_-$ keeps the $(n{+}1)$st row (the row indexed by $n{+}1$; the row indexed by $0$ is not included in $Y_-$).
	
	First write $\bar{\mathbf{x}}:=\dfrac{1}{n}\sum_{i=1}^n \mathbf{x}^{(k)}_i=\mathbf{c}_k-\dfrac{1}{n}\mathbf{v}_{n+1}$, then
	\[
	\mathbf{x}_r^{(k)}=-\mathbf{x}^{(k)}_{n+1}+2\bar{\mathbf{x}}
	=\mathbf{c}_k-\Bigl(1+\dfrac{2}{n}\Bigr)\mathbf{v}_{n+1}.
	\]
	Hence, for $i=1,\dots,n$,
	\[
	\mathbf{x}^{(k)}_i-\mathbf{x}_r^{(k)}
	=\mathbf{v}_i+\Bigl(1+\dfrac{2}{n}\Bigr)\mathbf{v}_{n+1},
	\qquad
	\mathbf{x}^{(k)}_{n+1}-\mathbf{x}_r^{(k)}
	=2\Bigl(1+\dfrac{1}{n}\Bigr)\mathbf{v}_{n+1}.
	\]
	Using $\mathbf{u}=\mathbf{v}_{n+1}/\delta_k$ and $\mathbf{v}_i^\top\mathbf{v}_{n+1}=-\delta_k^2/n$ for $i\le n$, we obtain
	\[
	\bigl(\mathbf{x}^{(k)}_i-\mathbf{x}_r^{(k)}\bigr)^\top \mathbf{u}
	=\dfrac{\mathbf{v}_i^\top\mathbf{v}_{n+1}}{\delta_k}
	+\Bigl(1+\dfrac{2}{n}\Bigr)\dfrac{\mathbf{v}_{n+1}^\top\mathbf{v}_{n+1}}{\delta_k}
	=-\dfrac{\delta_k}{n}+\Bigl(1+\dfrac{2}{n}\Bigr)\delta_k
	=\Bigl(1+\dfrac{1}{n}\Bigr)\delta_k,
	\]
	so
	\[
	Y_+P_-=\Bigl( \bigl(\mathbf{x}^{(k)}_i-\mathbf{x}_r^{(k)}\bigr)^\top \mathbf{u} \Bigr)_{i=1}^n
	=\Bigl(1+\dfrac{1}{n}\Bigr)\delta_k\,\mathbf{1}_n.
	\]
	Similarly,
	\[
	Y_-P_-=\bigl(\mathbf{x}^{(k)}_{n+1}-\mathbf{x}_r^{(k)}\bigr)^\top \mathbf{u}
	=2\Bigl(1+\dfrac{1}{n}\Bigr)\delta_k.
	\]
	
	Finally,
	\[
	M=\mathrm{diag}(\ell_+)\,Y_+P_-\,(Y_-P_-)^{-1}
	=\Bigl(\dfrac{2}{n}I_n\Bigr)\cdot
	\dfrac{\bigl(1+\dfrac{1}{n}\bigr)\delta_k\,\mathbf{1}_n}{2\bigl(1+\dfrac{1}{n}\bigr)\delta_k}
	=\dfrac{1}{n}\,\mathbf{1}_n\in\mathbb{R}^{n\times 1}.
	\]
	Therefore
	\[
	\mu_{i,n+1}=e_i^\top M=\dfrac{1}{n},
	\qquad
	\mu_{i0}=\ell_i-\mu_{i,n+1}
	=\dfrac{2}{n}-\dfrac{1}{n}
	=\dfrac{1}{n},\quad i=1,\dots,n,
	\]
	which are all non--negative.
	
	\subsubsection*{(ii) Centroid Query}
	
	Here $\mathbf{x}=\mathbf{c}_k$ and $\ell_i=\dfrac{1}{n+1}$ for $i=1,\ldots,n+1$, while
	$\ell_0=-1$. Thus $\mathcal{I}_+ = \{1,\ldots,n+1\}$ and
	$\mathcal{I}_-=\{0\}$. Hence $|\mathcal{I}_-|-1=0$ and $M$ is a
	$(n{+}1)\times 0$ empty matrix (no computation is needed). The only non--zero
	coefficients are
	\[
	\mu_{i0}=\ell_i=\dfrac{1}{n+1},\qquad i=1,\ldots,n+1 .
	\]
	
	\subsubsection*{(iii) Shrinking Step}
	
	Fix $i\in\{2,\ldots,n+1\}$ and $\mathbf{x}=\hat{\mathbf{x}}_i=\gamma \mathbf{x}^{(k)}_i+(1-\gamma)\mathbf{x}^{(k)}_1$,
	with
	$\ell_0=-1$, $\ell_1=1-\gamma$, $\ell_i=\gamma$, and $\ell_j=0$ otherwise.
	Again $\mathcal{I}_-=\{0\}$, so $M$ is empty and the (only) non--zero
	$\mu_{ij}$ are
	\[
	\mu_{10}=\ell_1=1-\gamma,\qquad
	\mu_{i0}=\ell_i=\gamma ,
	\]
	both non--negative.
	
	\subsection*{A.4. Convex Extrapolation Error for Reflection}\label{ap:a4}
	Recall that in convex settings, the maximum of $\dfrac{1}{2}\ip{G}{H}$ is $\dfrac{L}{2}\max\left(\tr(G_+), -\tr(G_-)\right)$.
	
	For reflection, the formulation of $G$ remains
	\[
	G=\dfrac{2(n+1)}{n^2}\,\delta_k^2\,I_n
	-\dfrac{2(n+1)(n+2)}{n^2}\,\delta_k^2\,\mathbf{u}\mathbf{u}^\top ,
	\]
	where the eigenvalues of $G$ are
	\[
	\underbrace{\lambda_+=\dfrac{2(n+1)}{n^2}\,\delta_k^2}_{\text{mult. }n-1},\qquad
	\lambda_-=-\dfrac{2(n+1)^2}{n^2}\,\delta_k^2 .
	\]
	Consequently, the desired error bound is
	\[
	\boxed{\;
		\Bigl|\,f(\mathbf{x}^{(k)}_r)+f(\mathbf{x}^{(k)}_{n+1})-\dfrac{2}{n}\sum_{i=1}^n f(\mathbf{x}^{(k)}_i)\Bigr|
		\ \le\ \dfrac{L}{2}\,\max\!\Bigl\{\dfrac{2(n^2-1)}{n^2},\ \dfrac{2(n+1)^2}{n^2}\Bigr\}\delta_k^2
		\ =\ L\left(1+\dfrac{1}{n}\right)^2\delta_k^2.\; }
	\]
	
	\bibliographystyle{unsrtnat}
	\bibliography{refs}
	
\end{document}